\documentclass[11pt]{amsart}
\usepackage[T1]{fontenc}
\usepackage[utf8]{inputenc}
\usepackage{microtype}
\usepackage{mathtools}
\usepackage{amssymb}
\usepackage{array}
\usepackage{booktabs}
\usepackage{longtable}
\usepackage{enumitem}
\usepackage[margin=1.1in,heightrounded]{geometry}
\usepackage[hidelinks]{hyperref}

\allowdisplaybreaks
\numberwithin{equation}{section}

\theoremstyle{plain}
\newtheorem{theorem}{Theorem}[section]
\newtheorem{proposition}[theorem]{Proposition}
\newtheorem{lemma}[theorem]{Lemma}
\newtheorem{corollary}[theorem]{Corollary}
\newtheorem{conjecture}[theorem]{Conjecture}

\theoremstyle{definition}

\newtheorem{example}[theorem]{Example}

\theoremstyle{remark}
\newtheorem{remark}[theorem]{Remark}

\newcommand{\A}{\mathcal A}
\newcommand{\B}{\mathcal B}
\newcommand{\CC}{\mathcal C}
\newcommand{\PP}{\mathbb P}
\newcommand{\Q}{\mathbb Q}
\newcommand{\C}{\mathbb C}
\newcommand{\F}{\mathbb F}
\newcommand{\LL}{\mathcal L}

\DeclareMathOperator{\AR}{AR}
\DeclareMathOperator{\mdr}{mdr}
\DeclareMathOperator{\indeg}{indeg}

\DeclarePairedDelimiter{\floor}{\lfloor}{\rfloor}
\DeclarePairedDelimiter{\ceil}{\lceil}{\rceil}

\title[A generalized Terao Conjecture for line arrangements]
{On a generalized Terao Conjecture for line arrangements}

\author{Alexandru Dimca}
\address{Universit\'e C\^ote d'Azur, CNRS, LJAD, France}
\address{Simion Stoilow Institute of Mathematics, Bucharest, Romania}
\email{Alexandru.DIMCA@univ-cotedazur.fr}

\author{Piotr Pokora}
\address{Department of Mathematics, UKEN Krakow, Podchor\c a\.zych 2,
PL-30-084 Krak\'ow, Poland}
\email{piotr.pokora@uken.krakow.pl}

\date{\today}

\begin{document}

\begin{abstract}
For a line arrangement $\A\colon f=0$ in $\PP^2$, let $\nu(\A)$ be the maximal
dimension of a graded piece of the Jacobian module of $f$. We study the
conjecture that $\nu(\A)$ depends only on the intersection lattice of $\A$.
It is known that $\A$ is free if and only if $\nu(\A)=0$, and hence this
conjecture is a strengthening of Terao's conjecture for line
arrangements. We show that the conjecture holds for arrangements of at most
$13$ lines, except possibly for arrangements of exactly $13$ lines whose
intersection points have maximal multiplicity $5$.
\end{abstract}

\maketitle


\section{Introduction}
\label{sec:intro}

Let $S=\C[x,y,z]$ and let $\A\colon f=0$ be an arrangement of $d$ lines in
$\PP^2$. We write
\[
J_f= \langle f_x,f_y,f_z \rangle,\qquad M(f)=S/J_f,
\]
and we denote by
\[
D_0(f)=\AR(f)=\bigl\{(a,b,c)\in S^3\ :\ af_x+bf_y+cf_z=0\bigr\}
\]
the module of Jacobian relations. The invariant of primary interest here is
\[
r=r(\A)=\mdr(f)=\indeg D_0(f),
\]
the minimal degree of a nonzero Jacobian relation. We let $m=m(\A)$ be the
maximal multiplicity of an intersection point of $\A$, and $n_k=n_k(\A)$ the
number of points where exactly $k$ lines of $\A$ meet. Every arrangement of
$d$ lines satisfies the following counts:
\begin{equation}\label{eq:nk}
\sum_{k\ge2}\binom k2 n_k=\binom d2,
\qquad
\tau(\A)=\sum_{k\ge2}(k-1)^2n_k,
\end{equation}
where $\tau(\A)$ is the global Tjurina number. The intersection lattice
$\LL(\A)$ determines $d$, all the $n_k$, hence $m(\A)$ and $\tau(\A)$. However, it does
not determine $r(\A)$ in general -- see Example \ref{ex:ziegler}.

Let $N(f)=I_f/J_f$ be the Jacobian module of $f$, where $I_f$ is the saturation
of $J_f$, and set
\[
\nu(\A)=\max_j\dim N(f)_j .
\]
By \cite[Theorem 1.2]{DimcaOpen}, we have 
\begin{equation}\label{eq:nu}
\nu(\A)=
\begin{cases}
(d-1)^2-r(d-r-1)-\tau(\A), & r<\dfrac{d-2}{2},\\[3mm]
\ceil*{\dfrac{3(d-1)^2}{4}}-\tau(\A), & r\ge\dfrac{d-2}{2}.
\end{cases}
\end{equation}
Our paper is devoted to the following question.

\begin{conjecture}\label{C1}
If $\A$ and $\A'$ are line arrangements in $\PP^2$ with isomorphic intersection lattices,
then $\nu(\A)=\nu(\A')$.
\end{conjecture}

Since $d$ and $\tau$ are lattice invariants, \eqref{eq:nu} shows that a
counterexample to Conjecture \ref{C1} is precisely a pair of line arrangements $\A,\A'$ with
\begin{equation}\label{eq:C2}
\LL(\A)\cong\LL(\A'),
\qquad
r(\A)<\frac{d-2}{2},
\qquad \text{ and }
r(\A)<r(\A').
\end{equation}
We refer throughout to such a pair as a \emph{counterexample}, and we always
write $r=r(\A)$, $r'=r(\A')$, $m=m(\A)=m(\A')$ and $d$ for the common degree.

The second condition in \eqref{eq:C2} has a transparent geometric meaning. Let
$(a,b)$ with $a\le b$ be the generic splitting type of the rank two bundle on
$\PP^2$ associated with $D_0(f)$, so that $a+b=d-1$. One always has $a\le r$, and more precisely
$a=\min\{r,\floor{(d-1)/2}\}$, by \cite[Theorem 1.1 and Proposition
3.2]{AbeDimcaSplit}. In particular, $a=r$ as soon as $r<(d-2)/2$. Hence
$r<(d-2)/2$ is equivalent to $b-a\ge2$, that is, to
the bundle being \emph{unbalanced}. This explains the shape of \eqref{eq:nu}: in the balanced
range the splitting type, and hence $\nu$, does not see $r$ at all. It also
explains why the classical examples of lattice-isomorphic arrangements with
distinct $r$ do not contradict Conjecture \ref{C1} -- in all such examples known
to us, starting with Ziegler's example (Example \ref{ex:ziegler}), both members
are balanced. Conjecture \ref{C1} is closely related to
\cite[Conjecture 3.9]{DimcaOpen}, and it is equivalent to a positive answer to
\cite[Question 7.12]{Cook+}, which asks whether the splitting type is a
combinatorial invariant. In the unbalanced range it can also be phrased in
terms of unexpected curves of the dual point configurations -- see
Remark \ref{rk:chmn}.

\subsection*{Results of the paper}

Our first group of results restricts the numerical invariants of a potential
counterexample to Conjecture \ref{C1}.

\begin{theorem}\label{thmA}
Let $\A,\A'$ be a counterexample as in \eqref{eq:C2}. Then
\[
4\le m\le r<\frac{d-2}{2},
\qquad\text{hence}\qquad
d\ge 2m+3 .
\]
If moreover $m=4$, then $r=\floor*{(d-3)/2}$. If $r=m$, then $d\le3m+1$. There
is no counterexample with $d\le12$, and none with $r=m=4$. Finally, for every
line $L\in\A$, either $L$ contains at most $r+1$ singular points of $\A$, or
deleting $L$ from $\A$ and the corresponding line from $\A'$ produces a
counterexample with $d-1$ lines and minimal degree of a Jacobian relation
$r-1$.
\end{theorem}

Theorem \ref{thmA} is assembled from Propositions \ref{P:basic} and \ref{P:rm},
Lemma \ref{lem:del} and Theorem \ref{thm:eleven} below. The last statement, the
\emph{deletion lemma}, is the main new tool of the paper: it allows us to
propagate the non-existence of counterexamples from one degree to the next. The
smallest degree left open by the first statements is $d=13$, with
$(m,r)\in\{(4,5),(5,5)\}$, and the major part of this paper concerns the case
$m=4$, $r=5$. Our main result there is the following.

\begin{theorem}\label{corB}
There is no counterexample to Conjecture \ref{C1} with $d=13$ and $m=4$.
Consequently, Conjecture \ref{C1} holds for all line arrangements of at most
$12$ lines, and for all arrangements $\A$ of $13$ lines with $m(\A)\ne5$.
\end{theorem}

By Theorem \ref{thmA} the case $d=13$ leaves only $m\in\{4,5\}$, so the sole
configuration not covered by Theorem \ref{corB} is $d=13$, $m=5$, $r=5$ -- see
Remark \ref{rk:m5}.

The proof of Theorem \ref{corB} rests on the following purely theoretical
result.

\begin{theorem}\label{thmB}
Let $\A,\A'$ be a counterexample as in \eqref{eq:C2} with $d=13$, $m=4$ and
$r=5$. Then every line of $\A$ contains at most six singular points of $\A$,
and
\[
105\ \le\ \tau(\A)\ \le\ 107,
\qquad
3\ \le\ n_4\ \le\ 3\tau(\A)-312 .
\]
\end{theorem}

Theorem \ref{thmB} follows from Proposition \ref{P:thirteen} and
Corollary \ref{cor:thmB}. It reduces Theorem \ref{corB} to
finitely many intersection lattices, namely those with
$(\tau(\A),n_4)\in\{(105,3)\}\cup\{106\}\times\{3,\dots,6\}\cup\{107\}\times\{3,\dots,9\}$
in which every line contains at most six singular points. These are classified
in Section \ref{sec:104}, and the classification is computer assisted; none of
the resulting lattices is realizable over $\C$.

\begin{remark}\label{rk:m5}
The case $d=13$, $m=5$, $r=5$ is admissible, since $m\le r<(d-2)/2$ and
$d=2m+3$, while Proposition \ref{P:rm} gives only $d\le3m+1=16$. The
theoretical part of our argument applies to it as well: by
Proposition \ref{P:m5}, a counterexample with $d=13$ and $m=5$ has
$\tau(\A)\in\{106,107\}$, every line of $\A$ contains at most six singular
points, and $(n_2,n_3,n_4,n_5)$ is one of fifteen explicit numerical types, seven of
which are excluded by Proposition \ref{P:m5bezout}.
The classification of Section \ref{sec:104}, however, enumerates blocks of
sizes three and four only. Extending it to blocks of size five would settle
this case, but we leave it open here.
\end{remark}

The next degrees to consider are $14$, $15$ and $16$. Here the deletion lemma
is particularly effective, since it can be fed with the results in degrees $13$,
$14$ and $15$ respectively.

\begin{theorem}\label{thmC}
Let $\A,\A'$ be a counterexample as in \eqref{eq:C2} with $d\in\{14,15,16\}$.
\begin{enumerate}[label=\textnormal{(\roman*)}]
\item If $d=14$, then $r=5$, $m\in\{4,5\}$, $\tau(\A)\in\{125,126\}$, every line
of $\A$ contains at most six singular points of $\A$, and $(n_2,\dots,n_m)$ is
one of the ten numerical types listed in Proposition \ref{P:fourteen}.
\item If $d=15$ and $r=5$, then $m=5$, $\tau(\A)=148$, $\A'$ is free with
exponents $(6,8)$, and $(n_2,\dots,n_5)$ is one of three numerical types. In
particular, such a counterexample contradicts Terao's Conjecture in degree
$15$. If $d=15$ and $r=6$, then every line of $\A$ contains at most
$154-\tau(\A)$ singular points of $\A$.
\item If $d=16$, then $r=6$. If moreover $m=4$, then $164\le\tau(\A)\le169$ and
every line of $\A$ contains at most seven singular points of $\A$. If $m\ge5$,
the same bound on the lines holds provided Terao's Conjecture holds in degree
$15$.
\end{enumerate}
The possible values of $\tau(\A)$ are listed in Tables \ref{tab:1516} and
\ref{tab:1516b}. In the three cases $(d,r,\tau(\A))=(15,5,148)$, $(15,6,147)$
and $(16,6,169)$ the arrangement $\A'$ is forced to be free, so that a
counterexample would contradict Terao's Conjecture in degree $d$.
\end{theorem}

Theorem \ref{thmC} is proved in Section \ref{sec:1516} as a compilation of
Propositions \ref{P:fourteen}, \ref{P:fifteen}, \ref{P:sixteen} and
\ref{P:terao}. Two further cases in which $\A'$ would be forced to be free,
$(d,r,\tau(\A))=(14,5,127)$ and $(16,5,171)$, do not occur at all.

Finally, we discuss the classical mechanism for producing lattice-isomorphic
arrangements with distinct $r$, recalled in Example \ref{ex:ziegler}, and show
that it cannot produce a counterexample in degrees $14, 15, 16$.

\begin{theorem}\label{thmD}
Let $\A$ be an arrangement of $d$ lines with $m(\A)\le6$ in which every line
passes through at most two points of multiplicity at least three. Then
$\tau(\A)\le121$ if $d=14$, $\tau(\A)\le141$ if $d=15$, and $\tau(\A)\le160$ if
$d=16$. If moreover $d=14$ and $m(\A)\le4$, then $\tau(\A)\le112$, and if
$d=16$ and $m(\A)\le5$, then $\tau(\A)\le156$. In particular, no such
arrangement realizes any of the numerical types of Theorem \ref{thmC}(i) or of
Table \ref{tab:1516}, nor any type of Table \ref{tab:1516b} with
$\tau(\A)>141$ if $d=15$, or $\tau(\A)>156$ if $d=16$.
\end{theorem}

We complement this by a dimension count. Writing $e(\A)$ for the expected
dimension of the realization space of $\LL(\A)$, we show in Section
\ref{sec:obstruction} that, among the numerical types left by
Theorem \ref{thmC}, the condition $e(\A)\ge1$ together with a purely
combinatorial incidence inequality is met by none of the ten types in degree
$14$, by $381$ of $1092$ types in degree $15$ and by $346$ of $1376$ types in
degree $16$. Since $e(\A)\ge1$ is sufficient but not necessary for a non-empty
realization space to be positive-dimensional, this does not exclude any type --
it only identifies the types whose non-empty realization spaces are
automatically positive-dimensional.

\subsection*{Conventions}

All arrangements are essential arrangements of pairwise distinct lines in
$\PP^2_\C$. For a line $L\in\A$ we write $n_k^L$ for the number of points of
multiplicity $k$ of $\A$ lying on $L$, and
\[
r_L=n_2^L+n_3^L+\dots+n_m^L
\]
for the total number of singular points of $\A$ on $L$. We abbreviate
$\B=\A\setminus\{L\}$. When $\A,\A'$ is a pair with
$\LL(\A)\cong\LL(\A')$, we fix such an isomorphism once and for all, write
$L'\in\A'$ for the line corresponding to $L\in\A$, and put
$\B'=\A'\setminus\{L'\}$. Then $\LL(\B)\cong\LL(\B')$, and $r_L=r_{L'}$.

\section{Numerical and homological preliminaries}
\label{sec:prelim}
We collect the facts used throughout the paper; proofs and references can be
found in \cite{DimcaBook,Dca,DimcaOpen}.

If $2m>d$ then
\begin{equation}\label{eq:d-m}
r=d-m ,
\end{equation}
and if $m=2$ then $\A$ is nodal, and
\begin{equation}\label{eq:nodal}
r=d-2 .
\end{equation}
If $m>2$ and $d\ge2m$ then
\begin{equation}\label{eq:bb}
m-1\le r\le d-m,
\qquad
r\ \ge\ \frac{2d}{m}-2 .
\end{equation}
Equality $r=m-1$ holds if and only if $\A$ is free with exponents $(m-1,d-m)$ -- 
see \cite[Theorem 1.10(2)]{DimcaDer} for the first inequality and for this
equality statement.

For positive integers $d,r$, we put
\[
\tau(d,r)_{\max}=(d-1)^2-r(d-r-1),
\qquad
\tau(d,r)_{\min}=(d-1)(d-r-1).
\]
The du Plessis--Wall inequalities from \cite{duPlessisWall} give us
\begin{equation}\label{eq:dpw}
\tau(d,r)_{\min}\ \le\ \tau(\A),
\qquad\text{and}\qquad
\tau(\A)\ \le\ \tau(d,r)_{\max}\ \text{ if }r<d/2,
\end{equation}
and for $r\ge d/2$ the sharper upper bound
\begin{equation}\label{eq:dpwp}
\tau(\A)\ \le\ \tau(d,r)_{\max}'
:=\tau(d,r)_{\max}-\binom{2r+2-d}{2}.
\end{equation}
The function $r\mapsto\tau(d,r)_{\max}$ for $r\le(d-1)/2$, continued by
$r\mapsto\tau(d,r)'_{\max}$ for $r\ge d/2$, is strictly decreasing on
$[0,d]$. Note that for $d$ odd and $r=(d-1)/2$ both expressions agree. Recall that
$\A$ is free if and only if $\nu(\A)=0$, and nearly free if and only if
$\nu(\A)=1$, see \cite{Dfree,DSnf}. By \eqref{eq:nu}, for $r<(d-2)/2$ this means
that equality in the upper bound of \eqref{eq:dpw} characterizes free
arrangements and $\tau(\A)=\tau(d,r)_{\max}-1$ characterizes nearly free ones.
Equality in \eqref{eq:dpwp} for $r\ge d/2$ characterizes \textit{maximal Tjurina}
arrangements of type $(d,r)$, see \cite{maxTjurina}.

We shall often use \eqref{eq:nu} in the following form. For every integer $k$
one has $\tau(d,k)_{\max}\ge(d-1)^2-(d-1)^2/4$, hence
$\tau(d,k)_{\max}\ge\ceil*{3(d-1)^2/4}$. Since $r\mapsto\tau(d,r)_{\max}$ is
decreasing for $r\le(d-1)/2$, \eqref{eq:nu} gives, for every integer
$s\le(d-1)/2$,
\begin{equation}\label{eq:taunu}
r(\A)\ge s\ \Longrightarrow\ \tau(\A)\ \le\ \tau(d,s)_{\max}-\nu(\A).
\end{equation}
In particular, if $r(\A)\ge s$ and $\A$ is neither free nor nearly free, then
$\tau(\A)\le\tau(d,s)_{\max}-2$.

If $\A$ is plus-one generated (POG for short) with exponents $(d_1,d_2,d_3)$, then
$d_1=r$ and $d_1+d_2=d$ by \cite[Theorem 2.3]{3syz}, so that the second exponent is $d_2=d-r$, and then we have
\begin{equation}\label{eq:d3}
\tau(\A)=(d-1)^2-r(d-r-1)-(d_3-d+r+1),
\end{equation}
while $\nu(\A)=d_3-d_2+1$, see \cite[Propositions 2.1 and 3.7]{3syz}. Moreover, one has
$d_2\le d_3\le d-2$, and the last inequality follows from \cite{Sch}.

We shall repeatedly use the following consequence of these facts.

\begin{lemma}\label{lem:rm}
Let $\A$ be an arrangement of $d$ lines with $m=m(\A)\ge3$, $d\ge2m+1$ and
$r(\A)=m$. Then $\A$ is either free with exponents $(m,d-m-1)$ or POG with
exponents $(m,d-m,d_3)$, where $d-m\le d_3\le d-2$.
\end{lemma}

\begin{proof}
Let $d_1\le d_2\le\cdots$ be the degrees of a minimal system of generators of
$D_0(f)$, so $d_1=r=m$, and let $\rho_1$ be a generator of degree $m$. Let $p$
be a point of multiplicity $m$, and let $\tilde D_p\in D_0(f)_{d-m}$ be the
local derivation associated with $p$ in \cite[Theorem 1.3]{DimcaDer}. By
\cite[Theorem 1.3(5)]{DimcaDer}, the three coefficients of $\tilde D_p$ vanish
simultaneously only at the multiple points of the subarrangement formed by the
lines of $\A$ not passing through $p$, that is at finitely many points of
$\PP^2$. If $D_0(f)_{d-m}=(S\rho_1)_{d-m}$, then $\tilde D_p=u\rho_1$ with
$\deg u=d-2m\ge1$, and the three coefficients of $\tilde D_p$ would vanish
along the curve $u=0$, a contradiction. Hence $d_2\le d-m$, and so
$d_1+d_2\le d$. On the other hand $d_1+d_2\ge d-1$, and $\A$ is free if and
only if $d_1+d_2=d-1$, and POG if and only if $d_1+d_2=d$, by
\cite[Theorem 2.3]{3syz}, see also \cite[Theorem 1.1]{ADP}. The bounds on $d_3$
are those recalled above. 
\end{proof}
\begin{remark}
The lemma above is also contained in \cite[Theorem 1.12(2)]{DimcaDer}.
\end{remark}
When $r>m$ the same argument only gives $d_1+d_2\le d+(r-m)$, so that $\A$ need
not be POG -- this is why the cases $r=m$ and $r>m$ are treated separately in
Section \ref{sec:1516}.

For a line $L\in\A$ and $\B=\A\setminus\{L\}$ the addition--deletion sequence
reads, for every integer $k$,
\begin{equation}\label{eq:AD}
0\longrightarrow D_0(f_\B)_{k-1}\longrightarrow D_0(f_\A)_k
\longrightarrow H^0\bigl(L,\mathcal O_L(k+1-r_L)\bigr)
\longrightarrow N(f_\B)_{k+d-3},
\end{equation}
see \cite{DIS,STY} and \cite[(2.6)]{ADP}, applied with $C_1=\B$ and $C_2=L$, see also
\cite[Theorem 1.8]{DimcaDer} for the corresponding exact sequence
$$0\to D_0(f_\B)(-1)\to D_0(f_\A)\to D_0(f''),$$ where $f''$ is the restriction
of $f_\B$ to $L$. Only the exactness of \eqref{eq:AD} at its first three terms
is used below.

We shall also use that
\begin{equation}\label{eq:mono}
r(\B)\le r(\A)\ \text{ for every subarrangement }\B\subset\A,
\qquad
r(\A)\le r(\B)+1\ \text{ if }\B=\A\setminus\{L\}.
\end{equation}
The second inequality is the injectivity in \eqref{eq:AD}. For the first, a
nonzero $\theta\in D_0(f_\A)_k$ is a logarithmic derivation of $\A$, hence of
$\B$, so $\theta(f_\B)=gf_\B$ with $g\in S_{k-1}$, and
$\theta-\frac{g}{\deg f_\B}E\in D_0(f_\B)_k$, where $E$ is the Euler
derivation. This element is nonzero, since $\theta(f_\A)=0\ne E(f_\A)$ shows
that $\theta$ is not a multiple of $E$.

We record next an elementary but useful invariance property, which will be
applied in Section \ref{sec:extremal}, where arrangements defined over
quadratic fields occur.

\begin{lemma}\label{lem:galois}
Let $K\subset\C$ be a subfield and let $\A\colon f=0$ be an arrangement of $d$
lines with $f\in K[x,y,z]$, all lines of $\A$ being defined over $K$.
\begin{enumerate}[label=\textnormal{(\roman*)}]
\item For every field extension $K'$ of $K$ and every $r$ one has
\[
\dim_{K'}\AR(f\otimes_KK')_r=\dim_K\AR(f)_r .
\]
In particular $r(\A)$ does not
depend on the field over which it is computed, and a Jacobian relation of
minimal degree may be chosen with all its coefficients in $K$.
\item Let $\sigma\colon K\to\C$ be a field embedding and let $\A^\sigma\colon
f^\sigma=0$ be the arrangement obtained by applying $\sigma$ to the coefficients
of the linear forms. Then $\LL(\A^\sigma)\cong\LL(\A)$ and
$r(\A^\sigma)=r(\A)$.
\end{enumerate}
\end{lemma}

\begin{proof}
(i) The graded piece $\AR(f)_r$ is the kernel of the $K$-linear map
$S_r^3\to S_{r+d-1}$ sending $(a,b,c)$ to $af_x+bf_y+cf_z$. In the monomial
bases this map is given by a matrix $M_r$ with entries in $K$, and the rank of a
matrix is unchanged by extension of the base field. Hence so is the dimension of
the kernel, and a $K$-basis of $\ker M_r$ remains a basis after extension.

(ii) Applying $\sigma$ to the coefficients of the linear forms of $\A$ carries
$M_r$ entrywise to the corresponding matrix for $f^\sigma$. As $\sigma$ is an
isomorphism of $K$ onto its image, it preserves the vanishing of minors, so
$\operatorname{rank}M_r^\sigma=\operatorname{rank}M_r$ and therefore
$\dim\AR(f^\sigma)_r=\dim\AR(f)_r$ for every $r$, whence $r(\A^\sigma)=r(\A)$
by (i). The same argument applied to the $2\times2$ and $3\times3$ minors of the coefficient
matrix of the linear forms shows that a set of lines of $\A$ is concurrent if
and only if the corresponding set of lines of $\A^\sigma$ is, and that distinct
lines remain distinct. Hence the two intersection lattices are isomorphic.
\end{proof}

\begin{remark}\label{rk:galois}
Lemma \ref{lem:galois}(ii) says that a pair of Galois-conjugate arrangements
can never provide a counterexample to Conjecture \ref{C1}: the two members have
isomorphic intersection lattices, but also equal $r$, hence equal $\nu$ by
\eqref{eq:nu}. This is worth keeping in mind below, since the two
lattices described in Section \ref{sec:extremal} admit exactly two
realizations up to projective equivalence, conjugate over a quadratic field,
and one might otherwise be tempted to regard such a pair as a potential
candidate.
\end{remark}

We shall use the strengthened Hirzebruch-type inequality (see for instance
\cite{PokoraLMY}):
if $n_k=0$ for every $k>2d/3$, then
\begin{equation}\label{eq:hirz}
n_2+\frac34 n_3\ \ge\ d+\sum_{k\ge5}\Bigl(\frac{k^2}{4}-k\Bigr)n_k .
\end{equation}

Finally, we record three elementary identities that will be used repeatedly.
Since $(k-1)^2-\binom k2=\binom{k-1}{2}$, combining the two formulas in
\eqref{eq:nk} gives us
\begin{equation}\label{eq:tauT}
\tau(\A)=\binom d2+\sum_{p}\binom{m_p-1}{2},
\end{equation}
where the sum runs over all singular points $p$ of $\A$, of multiplicity $m_p$, a
node contributing $0$. Thus $\tau(\A)$ is determined by the points of
multiplicity \textit{at least three} alone. Secondly, deleting a line $L$ lowers the
Tjurina number by $2k-3$ at each point of multiplicity $k$ lying on $L$, so we have
\begin{equation}\label{eq:delete}
\tau(\A)-\tau(\B)=\sum_{k\ge2}(2k-3)\,n_k^L ,
\qquad
\sum_{k\ge2}(k-1)\,n_k^L=d-1 ,
\end{equation}
and consequently
\begin{equation}\label{eq:deleterL}
\tau(\A)-\tau(\B)=2(d-1)-r_L .
\end{equation}
Thirdly, counting the incidences between the lines and the singular points of
$\A$, and using $k=2\binom k2-3\binom{k-1}{2}+\binom{k-2}{2}$ together with
\eqref{eq:nk} and \eqref{eq:tauT}, we get
\begin{equation}\label{eq:sumr}
\sum_{L\in\A}r_L=\sum_{k\ge2}k\,n_k
=5\binom d2-3\tau(\A)+\sum_{k\ge4}\binom{k-2}{2}n_k .
\end{equation}
The last sum equals $n_4+3n_5+6n_6+\ldots$, and it is non-negative and vanishes
exactly when $m(\A)\le3$.

\section{First restrictions on a counterexample}
\label{sec:first}

\begin{proposition}\label{P:basic}
Let $\A,\A'$ be a counterexample. Then $4\le m\le r<(d-2)/2$, hence
$d\ge2m+3$. If $m=4$, then $r=\floor*{(d-3)/2}$.
\end{proposition}

\begin{proof}
If $2m>d$ then $r=d-m$ by \eqref{eq:d-m} is a lattice invariant, so $r=r'$,
which is impossible and hence $d\ge2m$. If $m=2$ then $r=d-2$ by
\eqref{eq:nodal}, contradicting $r<(d-2)/2$. If $m=3$, then $d\ge6$ and
\eqref{eq:bb} and \eqref{eq:C2} give $2d/3-2\le r<(d-2)/2$, whence $d\le5$, a
contradiction. Hence $m\ge4$.

Next, $r\ge m-1$ always holds, with equality if and only if $\A$ is free with
exponents $(m-1,d-m)$. If $r=m-1$ then $\tau(\A)=\tau(d,r)_{\max}$, while
$r'\ge r+1$ and $r+1\le(d-1)/2$ give
$\tau(\A')\le\tau(d,r+1)_{\max}<\tau(d,r)_{\max}$ by \eqref{eq:taunu}, so that
$\tau(\A)>\tau(\A')$, contradicting $\LL(\A)\cong\LL(\A')$. Hence $m\le r$, and
$m\le r<(d-2)/2$ yields $d\ge2m+3$.

Finally, let $m=4$. Then \eqref{eq:bb} gives $d/2-2\le r<(d-2)/2$. If $d=2d'+1$, then
the unique integer solution is $r=d'-1$, and if $d=2d'$ it is $r=d'-2$. Both
are $\floor*{(d-3)/2}$.
\end{proof}

The same mechanism, applied one step higher, constrains the case $r=m$.

\begin{proposition}\label{P:rm}
Let $\A,\A'$ be a counterexample with $r=m$. Then $d\le3m+1$. Moreover, there is
no counterexample with $r=m=4$, and none with $d\le12$.
\end{proposition}

\begin{proof}
Since $r=m$, Lemma \ref{lem:rm} shows that either $\A$ is free with exponents
$(m,d-m-1)$, which is excluded exactly as in the proof of Proposition
\ref{P:basic}, or $\A$ is POG with exponents $(m,d-m,d_3)$. In the latter case \eqref{eq:d3} gives
\begin{equation}\label{eq:TA}
\tau(\A)=(d-1)^2-m(d-m-1)-(d_3-d+m+1).
\end{equation}
On the other hand $m=r<(d-2)/2$ implies $m+1\le(d-1)/2$, so that $r'\ge m+1$
and \eqref{eq:taunu} give
\begin{equation}\label{eq:TAp}
\tau(\A')\le\tau(d,m+1)_{\max}-\nu(\A').
\end{equation}
Expanding, $\tau(d,m+1)_{\max}-\tau(\A)=d_3-2d+3m+3$, so \eqref{eq:TAp} and
$\tau(\A)=\tau(\A')$ force $d_3\ge2d-3m-3+\nu(\A')$. Since $d_3\le d-2$ by \cite{Sch},
we obtain $d-2\ge2d-3m-3$, that is $d\le3m+1$.

Suppose now $r=m=4$. By Proposition \ref{P:basic} we have
$4=r=\floor*{(d-3)/2}$, so $d\in\{11,12\}$, and $d=11$ is excluded by
Theorem \ref{thm:eleven} below. For $d=12$ we know that Terao's Conjecture and
its nearly free analogue hold, by \cite{BK} and \cite{DIM2}: if $\A'$ were free
or nearly free, then so would be $\A$, with the same exponents, contradicting
$r<r'$. Hence $\nu(\A')\ge2$, and $d_3\le d-2$ yields $d\le3m-1=11$, a
contradiction.

Finally, if $d\le12$ then Proposition \ref{P:basic} forces $m=4$ and hence
$r=4=m$, which has just been excluded.
\end{proof}

\begin{lemma}\label{lem:nofree}
Let $\A,\A'$ be a counterexample with $d\le15$. Then, for
every line $L\in\A$, the deleted arrangement $\B=\A\setminus\{L\}$ is not free.
\end{lemma}

\begin{proof}
Suppose that $\B$ is free with exponents $(d_1,d_2)$, $d_1\le d_2$, so that
$d_1+d_2=d-2$, since $\B$ consists of $d-1$ lines. By the addition theory for free arrangements,
see \cite[Theorem 1.4]{POG}, the arrangement $\A=\B\cup\{L\}$ is free with
exponents $(d_1,d_2+1)$ if $|\B^L|=d_1+1$, free with exponents $(d_1+1,d_2)$ if
$|\B^L|=d_2+1$, and POG with exponents $(d_1+1,d_2+1,|\B^L|-1)$ otherwise. In
all three cases
\[
r(\A)\in\{d_1,\,d_1+1\},
\]
and which of the three occurs is decided by $d_1$, $d_2$ and $|\B^L|$ alone.

Let $L'\in\A'$ and $\B'=\A'\setminus\{L'\}$ be as in the Conventions, so that
$\LL(\B)\cong\LL(\B')$. As $\deg\B=d-1\le14$, Terao's Conjecture \cite{BK}
makes $\B'$ free with the same exponents $(d_1,d_2)$. Moreover $|\B^L|=r_L$ is
the number of singular points of $\A$ lying on $L$, hence a lattice invariant,
so $|\B'^{L'}|=|\B^L|$. Therefore $\A$ and $\A'$ fall into the same case of the
trichotomy above and have the same exponents. In particular, $r(\A)=r(\A')$,
contradicting $r(\A)<r(\A')$.
\end{proof}

\subsection{The deletion lemma}
The following observation drives most of the paper. It says that a
counterexample either restricts the number of singular points on each of its
lines, or produces a counterexample with one line less.

\begin{lemma}\label{lem:del}
Let $\A,\A'$ be a counterexample, let $L\in\A$,
$\B=\A\setminus\{L\}$, and let $\B'=\A'\setminus\{L'\}$. Then
$r(\B)\in\{r-1,r\}$ and $r(\B')\ge r$. Moreover
\begin{enumerate}[label=\textnormal{(\roman*)}]
\item if $r(\B)=r$, then $r_L\le r+1$;
\item if $r(\B)=r-1$, then $\B,\B'$ is a counterexample as in \eqref{eq:C2},
of degree $d-1$, with $r(\B)=r-1$ and $m(\B)\le m$.
\end{enumerate}
In particular, if there is no counterexample of degree $d-1$ with minimal
degree of a Jacobian relation $r-1$ and maximal multiplicity at most $m$, then
every line of $\A$ contains at most $r+1$ singular points of $\A$.
\end{lemma}

\begin{proof}
By \eqref{eq:mono} we have $r(\B)\le r\le r(\B)+1$ and
$r(\B')\ge r(\A')-1\ge r$.

(i) If $r(\B)=r$, then $D_0(f_\B)_{r-1}=0$, so \eqref{eq:AD} with $k=r$ embeds
$D_0(f_\A)_r\ne0$ into $H^0\bigl(L,\mathcal O_L(r+1-r_L)\bigr)$. Hence this
space is non-zero, that is $r_L\le r+1$.

(ii) We have $\LL(\B)\cong\LL(\B')$, and $\B$ is essential, since
$m(\B)\le m<d-1$. Moreover, $r(\B)=r-1<(d-2)/2-1<\bigl((d-1)-2\bigr)/2$, and $r(\B)=r-1<r\le r(\B')$.
Hence $\B,\B'$ satisfies \eqref{eq:C2} in degree $d-1$.
\end{proof}

We shall also use the following two consequences of the results of
Section \ref{sec:prelim}.

\begin{lemma}\label{lem:twoline}
Let $\A,\A'$ be a counterexample with $d\ge12$ and
$m\le5$. Then there are no two lines $L_1,L_2\in\A$ such that every point of
multiplicity $4$ of $\A$ lies on $L_1\cup L_2$ and every point of
multiplicity $5$ of $\A$ is the point $L_1\cap L_2$. In particular, $n_4\ge3$
if $m=4$, and $(n_4,n_5)\ne(0,1)$ if $m=5$.
\end{lemma}

\begin{proof}
Suppose that such lines exist and put $\CC=\A\setminus\{L_1,L_2\}$, an
arrangement of $d-2\ge10$ lines. A point of multiplicity $4$ of $\A$ loses at
least one line, and a point of multiplicity $5$ loses two, so $m(\CC)\le3$. If
$m(\CC)=2$, then $r(\CC)=d-4$ by \eqref{eq:nodal}, and if $m(\CC)=3$, then
$r(\CC)\ge2(d-2)/3-2$ by \eqref{eq:bb}. As $d\ge12$, in both cases
$r(\CC)>(d-3)/2\ge r$. However, $r(\CC)\le r$ by \eqref{eq:mono}, a
contradiction.

If $m=4$ and $n_4\le2$, then choosing a line of $\A$ through each point of
multiplicity four, and adding an arbitrary further line if these choices give
only one line, we obtain two lines containing all of them. If $m=5$ and
$(n_4,n_5)=(0,1)$, any two lines through the point of multiplicity five will
do.
\end{proof}

\begin{lemma}\label{lem:deltau}
Let $\A,\A'$ be a counterexample with $d\le15$. Then, for
every line $L\in\A$,
\[
\tau(\A\setminus\{L\})\ \le\ \tau(d-1,r)_{\max}-1,
\qquad\text{equivalently}\qquad
r_L\ \le\ 2d-3+\tau(d-1,r)_{\max}-\tau(\A).
\]
\end{lemma}

\begin{proof}
By Lemma \ref{lem:del} we have $r(\B')\ge r$, and
$r\le(d-3)/2\le\bigl((d-1)-1\bigr)/2$, so \eqref{eq:taunu} applied in degree
$d-1$ gives $\tau(\B)=\tau(\B')\le\tau(d-1,r)_{\max}-\nu(\B')$. If $\nu(\B')=0$,
then $\B'$ is free, hence so is $\B$ by Terao's Conjecture in degree
$d-1\le14$ (\cite{BK}), contradicting Lemma \ref{lem:nofree}. Hence
$\nu(\B')\ge1$, and the second form follows from \eqref{eq:deleterL}.
\end{proof}

Finally, Ziegler's multirestriction gives a lower bound for $r(\A)$ in terms of
the singular points on a single line.

\begin{lemma}\label{lem:ziegler}
Let $\A$ be an arrangement of $d$ lines and $L\in\A$. Let $\A^L$ be the set of
singular points of $\A$ on $L$, let $m^L(p)=m_p-1$ for $p\in\A^L$, and let
$d_1^L\le d_2^L$ be the exponents of the Ziegler multirestriction
$(\A^L,m^L)$, so that $d_1^L+d_2^L=d-1$. Then
\[
r(\A)\ \ge\ d_1^L .
\]
In particular, if $d\ge5$ and every singular point of $\A$ on $L$ is a triple
point, then $r(\A)\ge(d-1)/2$. Hence no line of a counterexample $\A$ as in
\eqref{eq:C2} contains only triple points of $\A$.
\end{lemma}

\begin{proof}
Let $\alpha_L$ be a linear form defining $L$, let $D(\A)$ be the module of
logarithmic derivations of $\A$, $\theta_E$ the Euler derivation, and
$D_L(\A)=\{\theta\in D(\A)\ :\ \theta(\alpha_L)=0\}$. For $\theta\in D(\A)$ we
have $\theta(\alpha_L)=g\alpha_L$ for some $g\in S$, and
$\theta-g\theta_E\in D_L(\A)$, so $D(\A)=S\theta_E\oplus D_L(\A)$. In the same
way $D(\A)=S\theta_E\oplus D_0(f)$, where $D_0(f)$ is viewed as the module of
derivations annihilating $f$. Hence $D_L(\A)\cong D(\A)/S\theta_E\cong D_0(f)$
as graded $S$-modules, and the initial degree of $D_L(\A)$ is $r(\A)$.

By \cite{ZieMulti}, reducing the coefficients modulo $\alpha_L$ defines a
degree preserving $S$-linear map
\[
\pi\colon D_L(\A)\longrightarrow D(\A^L,m^L),
\]
where $D(\A^L,m^L)$ is a free module over $S/\alpha_LS$ with a basis of degrees
$d_1^L,d_2^L$. If $\pi(\theta)=0$, then $\theta=\alpha_L\theta'$ for a
derivation $\theta'$, and $\theta'\in D_L(\A)$, since $\alpha_L$ is prime to
the linear forms of the other lines of $\A$. Thus $\ker\pi=\alpha_LD_L(\A)$.
Now let $0\ne\theta\in D_L(\A)_{r(\A)}$. As $D_L(\A)_{r(\A)-1}=0$, we have
$\theta\notin\ker\pi$, so $\pi(\theta)$ is a nonzero element of degree $r(\A)$
of $D(\A^L,m^L)$, and therefore $r(\A)\ge d_1^L$.

If every singular point on $L$ is a triple point, then $|\A^L|=(d-1)/2\ge2$ and
$m^L\equiv2$, so $(\A^L,m^L)$ has exponents $\bigl((d-1)/2,(d-1)/2\bigr)$ by
\cite{Wakamiko}, see also \cite[Proposition 1.23]{Yoshinaga}. Hence
$r(\A)\ge(d-1)/2$, which is incompatible with $r(\A)<(d-2)/2$.
\end{proof}

\section{Degree eleven}
\label{sec:eleven}

\begin{proposition}\label{P:eleven}
Let $\A,\A'$ be a counterexample with $d=11$ and $m=4$. Then $\A$ is POG with
exponents $(4,7,9)$, and $\A'$ is either a minimal POG arrangement with
exponents $(5,6,7)$ or a maximal Tjurina arrangement of type $(11,6)$, with
exponents $(6,6,6,6)$. In particular, one has $\tau(\A)=\tau(\A')=73$. Moreover,
$r_L\le5$ for every line $L\in\A$.
\end{proposition}

\begin{proof}
By Proposition \ref{P:basic} we have $r=4=m$, so by Lemma \ref{lem:rm} either
$\A$ is free, which is excluded exactly as in the proof of
Proposition \ref{P:basic} (or by Terao's Conjecture in degree $11$ \cite{BK}),
or $\A$ is POG with
exponents $(4,7,d_3)$ and $7\le d_3\le9$. By \eqref{eq:d3}, we have
$\tau(\A)=82-d_3$.

Since $r'\ge5$, and since $\A'$ is neither free nor nearly free by \cite{BK}
and \cite[Theorem 4.9]{DIM2}, \eqref{eq:taunu} with $s=5$ gives
$\tau(\A')\le\tau(11,5)_{\max}-2=73$. Hence $82-d_3\le73$, so $d_3=9$ and
$\tau(\A)=\tau(\A')=73$. If $r'=5$ then $\A'$ is a minimal POG arrangement with
exponents $(5,6,7)$ by \cite[Theorem 1.5]{DimcaSticlaruMPOG}. If $r'\ge6$ then
$\tau(\A')\le\tau(11,6)_{\max}'=73$ by \eqref{eq:dpwp}, with equality only for
$r'=6$ and $\A'$ being maximal Tjurina with exponents $(6,6,6,6)$. Since
$\tau(\A')=73$, this is exactly the case.

For the last assertion, note that there is no counterexample of degree $10$,
since $d\ge2m+3\ge11$ by Proposition \ref{P:basic}. Hence case (ii) of
Lemma \ref{lem:del} cannot occur, and case (i) gives $r_L\le r+1=5$ for every
line $L\in\A$.
\end{proof}

\begin{theorem}\label{thm:eleven}
There is no counterexample with $d=11$ and $m=4$.
\end{theorem}

\begin{proof}
Assume such a pair exists. By Proposition \ref{P:eleven}, $\tau(\A)=73$ and
$r_L\le5$ for every $L$. From \eqref{eq:nk} with $d=11$ and $m=4$,
\[
n_2+3n_3+6n_4=55,
\qquad
n_2+4n_3+9n_4=73,
\]
whence $n_3+3n_4=18$, so $n_3=18-3n_4$ and $n_2=1+3n_4$. Counting the
incidences between lines and singular points in two ways gives us
\[
\sum_{L\in\A}r_L=2n_2+3n_3+4n_4=2(1+3n_4)+3(18-3n_4)+4n_4=56+n_4 .
\]
Since $m=4$ we have $n_4\ge1$, so $\sum_Lr_L\ge57$, while $r_L\le5$ on each of
the eleven lines gives $\sum_Lr_L\le55$. This contradiction proves the theorem.
\end{proof}

\section{Thirteen lines}
\label{sec:thirteen}

We begin with the part of the argument that does not depend on $m$.

\begin{proposition}\label{P:thirteen}
Let $\A,\A'$ be a counterexample as in \eqref{eq:C2} with $d=13$. Then $r=5$,
$m\in\{4,5\}$, $\tau(\A)\le107$, and every line of $\A$ contains at most six
singular points of $\A$.
\end{proposition}

\begin{proof}
By Propositions \ref{P:basic} and \ref{P:rm} we have $4\le m\le r<11/2$, and
since $r=m=4$ is impossible, $r=5$ and $m\in\{4,5\}$. Since $r'\ge6$,
\eqref{eq:taunu} with $s=6$ gives
$\tau(\A)=\tau(\A')\le\tau(13,6)_{\max}-\nu(\A')=108-\nu(\A')$. If
$\nu(\A')=0$, then $\A'$ is free, and Terao's Conjecture in degree $13$ \cite{BK} makes $\A$ free, which is impossible since $\nu(\A)\ne\nu(\A')$.
Hence $\nu(\A')\ge1$ and $\tau(\A)\le107$.

By Proposition \ref{P:rm} there is no counterexample with $12$ lines, so case
(ii) of Lemma \ref{lem:del} cannot occur, and case (i) gives $r_L\le r+1=6$
for every line $L\in\A$.
\end{proof}

\subsection{The case \texorpdfstring{$m=4$}{m=4}}
Throughout this subsection and the next one, $\A,\A'$ is a counterexample to
Conjecture \ref{C1} with
\[
d=13,\qquad m=4,\qquad r=5 .
\]
From \eqref{eq:nk} we get
\begin{equation}\label{eq:n234}
n_3=\tau(\A)-78-3n_4,
\qquad \text{ and } \qquad
n_2=312-3\tau(\A)+3n_4=3\bigl(n_4-\tau(\A)+104\bigr).
\end{equation}
For a line $L\in\A$ we set
\[
D_L=n_3^L+2n_4^L ,
\]
so that the relation $n_2^L+2n_3^L+3n_4^L=12$ of \eqref{eq:delete} becomes
\begin{equation}\label{eq:nL}
n_2^L=12-2D_L+n_4^L ,
\qquad
n_3^L=D_L-2n_4^L\ \ge\ 0 ,
\end{equation}
and $r_L=12-D_L$. By Proposition \ref{P:thirteen} we have $r_L\le6$, that is
$D_L\ge6$, and then \eqref{eq:nL} gives
\begin{equation}\label{eq:heavy}
n_2^L=n_4^L+12-2D_L\ \le\ n_4^L\ \le\ 4 ,
\end{equation}
the last inequality because $3n_4^L\le12$. In particular $6\le D_L\le8$ for
every line. Finally, summing $D_L$ over all lines, we get
\begin{equation}\label{eq:sumD}
\sum_{L\in\A}D_L=3n_3+8n_4 .
\end{equation}

\begin{corollary}\label{cor:thmB}
One has $105\le\tau(\A)\le107$ and $3\le n_4\le3\tau(\A)-312$. Hence
$(\tau(\A),n_4)$ is $(105,3)$, or $(106,n_4)$ with $3\le n_4\le6$, or
$(107,n_4)$ with $3\le n_4\le9$, and every line of $\A$ satisfies
$6\le D_L\le8$. Moreover, no two lines of $\A$ contain all the points of
multiplicity four.
\end{corollary}

\begin{proof}
The last statement and $n_4\ge3$ are Lemma \ref{lem:twoline}, and
$\tau(\A)\le107$ is Proposition \ref{P:thirteen}. By \eqref{eq:sumr} and
Proposition \ref{P:thirteen},
\[
390-3\tau(\A)+n_4=\sum_{L\in\A}r_L\le 6\cdot13=78,
\]
that is $n_4\le3\tau(\A)-312$. Together with $n_4\ge3$ this gives
$\tau(\A)\ge104+n_4/3\ge105$. The bounds on $D_L$ were noted above.
\end{proof}

Proposition \ref{P:thirteen} and Corollary \ref{cor:thmB} prove
Theorem \ref{thmB}.

\begin{remark}\label{rk:nolight}
It is instructive to analyse directly a line $L$ with
$r(\A\setminus\{L\})=4$. Using Lemmas \ref{lem:rm} and \ref{lem:nofree},
\cite{AIM,DIM2} and Terao's Conjecture in degree $12$, one finds that
$\B=\A\setminus\{L\}$ is POG with exponents $(4,8,10)$, so that $\tau(\B)=90$,
$\nu(\B)=3$ and, by \eqref{eq:deleterL}, $r_L=114-\tau(\A)\ge7$. On the other
hand $r(\B')\ge5=(12-2)/2$ gives $\nu(\B')=\ceil{3\cdot11^2/4}-90=1$ by
\eqref{eq:nu}. Thus $\B,\B'$ would be a counterexample in degree $12$, as
predicted by Lemma \ref{lem:del}, which is impossible by
Proposition \ref{P:rm}. The arrangements of Section \ref{sec:extremal}
show that lines with $r_L=114-\tau(\A)$ do occur in realizable lattices with
$d=13$, $m=4$ and $\tau\in\{106,107\}$.
\end{remark}

\subsection{The classification}
\label{sec:104}

By Corollary \ref{cor:thmB}, Theorem \ref{corB} follows from the next result.

\begin{theorem}\label{thm:class}
Let $\LL$ be the intersection lattice of an arrangement of thirteen lines with
$m=4$ such that $(\tau,n_4)$ is one of the pairs of Corollary \ref{cor:thmB},
every line satisfies $6\le D_L\le8$, and no two lines contain all the points of
multiplicity four. Then $\LL$ is not realizable over $\C$.
\end{theorem}

The proof is computer assisted -- all computations can be verified with the
programs and data in \cite{Repo}. We first reformulate the problem in dual
terms, then describe the reduction, the enumeration and the realizability
test, and finally record the outcome.

\subsection*{The dual formulation}
It is convenient to dualize the problem, using the language of abstract
configurations. The thirteen lines of $\A$ become thirteen points of $\PP^2$,
and a point of multiplicity $k$ becomes a line containing exactly $k$ of them,
that is a \emph{block} of size $k$. As $m=4$, only blocks of sizes $3$ and $4$
occur. Two blocks meet in at most one point, and two points lie on at most one
common block. For the point $p$ dual to a line $L$ we write
\[
t_p=n_3^L,\qquad q_p=n_4^L,\qquad e_p=n_2^L=12-2t_p-3q_p,\qquad
D_p=D_L=t_p+2q_p ,
\]
so that every point satisfies $6\le D_p\le8$ and $0\le e_p\le q_p\le4$ by
\eqref{eq:heavy}. By \eqref{eq:n234}, $n_3=\tau-78-3n_4$ and
$n_2=312-3\tau+3n_4$.

\subsection*{Numerical reduction}
\label{sub:num}
We enumerate the multisets of profiles $(t_p,q_p)$, $0\le p\le12$, compatible
with
\[
\sum_pt_p=3n_3,\qquad \sum_pq_p=4n_4,\qquad 6\le t_p+2q_p\le8,
\qquad 2t_p+3q_p\le12,
\]
and with the three counting conditions
\begin{equation}\label{eq:count}
\sum_p\binom{q_p}{2}\le\binom{n_4}{2},
\qquad
\sum_p\binom{t_p}{2}\le\binom{n_3}{2},
\qquad
q_{(1)}+q_{(2)}\le n_4+1,
\end{equation}
where $q_{(1)}\ge q_{(2)}$ are the two largest values of $q_p$. The first two
express that two blocks meet in at most one point. The third expresses that two
points lie on at most one common block: if $q_p+q_{p'}>n_4+1$ the blocks
through $p$ and those through $p'$ would share at least two members. The
relation $\sum_pe_p=2n_2$ is then automatic. Finally, Lemma
\ref{lem:twoline} constrains the blocks of size four through the following
elementary observation.

\begin{lemma}\label{lem:qbound}
For every point $p$ one has $q_p\le n_4-2$.
\end{lemma}

\begin{proof}
If $q_p\ge n_4-1$, the line dual to $p$ carries all but at most one of the
points of multiplicity four. Together with any line of $\A$ through the
remaining one, it carries all of them, contradicting the hypothesis of
Theorem \ref{thm:class}, which holds for $\LL(\A)$ by
Lemma \ref{lem:twoline}.
\end{proof}

The numbers of profile multisets satisfying all these conditions are listed in
Table \ref{tab:extremal}.

\subsection*{The enumeration}
\label{sub:method}
The search proceeds in two stages, the blocks of size four first.

For the first stage we do not search at all. A family of $n_4$ blocks of size
four on thirteen points, pairwise meeting in at most one point, is determined
up to isomorphism by the family of subsets of the block set cut out by the
points lying on two or more blocks. Indeed, distinct points give distinct
subsets, two subsets meet in at most one block, each block contains at most
four of these points, and the remaining points of a block are private to it.
Enumerating such families of subsets of an $n_4$-element set with prescribed
sizes and reducing them modulo the symmetric group is a small computation,
and it returns the isomorphism classes directly rather than the very large
number of their labelled copies. Those classes on which two points meet every
block are discarded, by Lemma \ref{lem:twoline}.

For the second stage we fix a class from the first, assign the triple degrees
$t_p$ to the labels in all ways compatible with a given profile multiset, and
search for the blocks of size three with the blocks of size four held fixed.
Blocks are generated at their lexicographically smallest pair, so that every
labelled solution is produced exactly once, and the prescribed profiles make
the pruning tight. The solutions are then reduced modulo isomorphism by a
canonical form of the coloured incidence graph. No search was capped.

\subsection*{The realizability test}
Each class is then tested for realizability. We fix a projective frame on four
points no three of which lie in a common block, and propagate the incidences: a
line is determined as soon as two of its points are known, and a point as soon
as it lies on two known lines. When the propagation stops, an unknown point
$p$ lying on a known line through two already placed points $A$ and $B$ is
written as $A+sB$ with a new parameter $s$. This chart misses only the point
$B$, which is distinct from $p$ in any realization, so no realization is lost.
Points lying on no known line are never parametrized: if such points remain,
another frame is used. Every incidence not consumed by the propagation becomes
an equation, which is a necessary condition. Consequently, if the resulting
ideal is the unit ideal, the class is not realizable over a field of
characteristic zero. If the ideal is proper and zero-dimensional we solve it
exactly, over the relevant number field when its solutions are not rational,
construct the configuration explicitly, and verify that no point degenerates
to $(0:0:0)$, that no two points coincide, and that the recovered lattice is
the intended one. Only these two outcomes, a unit ideal and a zero-dimensional
ideal all of whose solutions fail one of these tests, are used as proofs of
non-realizability.

\subsection*{Outcome}
For $(n_2,n_3,n_4)=(6,18,3)$, Lemma \ref{lem:qbound} gives $q_p\le1$, so the
three blocks of size four are pairwise disjoint, and the profile multiset is
unique. The first stage has a single class, and the second stage produces
$41472$ labelled solutions forming two isomorphism classes, both with unit
realization ideal. The remaining cases are summarized in
Table \ref{tab:extremal} -- in each of them every class has unit realization
ideal.

The rows of Table \ref{tab:extremal} with no profile can be checked by hand.
For $(n_2,n_3,n_4)=(3,19,3)$ and $(0,20,3)$, Lemma \ref{lem:qbound} gives
$q_p\le1$, so twelve points have $(t_p,q_p)=(4,1)$ and one has $(6,0)$, whence
$\sum_pt_p=54\ne3n_3$. For $(18,2,9)$ one has $\sum_pD_p=3n_3+8n_4=78$, so
$D_p=6$ and $t_p=6-2q_p$ is even for every $p$. Then
$\sum_p\binom{t_p}{2}\le\binom22=1$ forces $\sum_pt_p\le2<6=3n_3$.

\begingroup
\small
\begin{longtable}{ccccc}
\caption{The lattices of Theorem \ref{thm:class}.}
\label{tab:extremal}\\
\toprule
$\tau$ & $(n_2,n_3,n_4)$ & profiles & classes & realizable\\
\midrule
\endfirsthead
\toprule
$\tau$ & $(n_2,n_3,n_4)$ & profiles & classes & realizable\\
\midrule
\endhead
\bottomrule
\endlastfoot
$105$ & $(6,18,3)$ & $1$ & $2$ & none\\
\midrule
$106$ & $(3,19,3)$ & $0$ & $0$ & ---\\
$106$ & $(6,16,4)$ & $4$ & $4$ & none\\
$106$ & $(9,13,5)$ & $16$ & $24$ & none\\
$106$ & $(12,10,6)$ & $10$ & $15$ & none\\
\midrule
$107$ & $(0,20,3)$ & $0$ & $0$ & ---\\
$107$ & $(3,17,4)$ & $2$ & $0$ & ---\\
$107$ & $(6,14,5)$ & $16$ & $29$ & none\\
$107$ & $(9,11,6)$ & $31$ & $157$ & none\\
$107$ & $(12,8,7)$ & $22$ & $46$ & none\\
$107$ & $(15,5,8)$ & $2$ & $3$ & none\\
$107$ & $(18,2,9)$ & $0$ & $0$ & ---\\
\end{longtable}
\endgroup

This proves Theorem \ref{thm:class}, and together with Theorem \ref{thmB} it
proves Theorem \ref{corB}: \textbf{there is no counterexample with} $d=13$ and
$m=4$.

\subsection{The case \texorpdfstring{$m=5$}{m=5}}
\label{sub:m5}
The theoretical part of the argument extends to $m=5$.

\begin{proposition}\label{P:m5}
Let $\A,\A'$ be a counterexample as in \eqref{eq:C2} with $d=13$ and $m=5$.
Then $r=5$, $\A$ is POG with exponents $(5,8,d_3)$ where $d_3\in\{9,10\}$,
$\tau(\A)=116-d_3\in\{106,107\}$, every line of $\A$ contains at most six
singular points of $\A$, and
\[
n_4+3n_5\ \le\ 3\tau(\A)-312 .
\]
More precisely, $(n_2,n_3,n_4,n_5)$ is one of the four types
\[
(5,19,1,1),\quad (8,16,2,1),\quad (11,13,3,1),\quad (10,16,0,2)
\]
if $\tau(\A)=106$, or one of the eleven types
\[
\begin{gathered}
(2,20,1,1),\ (5,17,2,1),\ (8,14,3,1),\ (11,11,4,1),\ (14,8,5,1),\ (17,5,6,1),\\
(7,17,0,2),\ (10,14,1,2),\ (13,11,2,2),\ (16,8,3,2),\ (15,11,0,3)
\end{gathered}
\]
if $\tau(\A)=107$.
\end{proposition}

\begin{proof}
By Proposition \ref{P:thirteen}, $r=5$, $\tau(\A)\le107$ and $r_L\le6$ for
every line. Since $r=m=5$ and $d=13\ge2m+1$, Lemma \ref{lem:rm} shows that
$\A$ is either free with exponents $(5,7)$, which is excluded exactly as in the
proof of Proposition \ref{P:basic}, or POG with exponents $(5,8,d_3)$, where
$8\le d_3\le11$, so that $\tau(\A)=116-d_3$ by \eqref{eq:d3}. By
\eqref{eq:sumr},
\[
390-3\tau(\A)+n_4+3n_5=\sum_{L\in\A}r_L\le78 ,
\]
which is the stated inequality. As $n_5\ge1$, it gives $\tau(\A)\ge105$, and
$\tau(\A)=105$ would force $(n_4,n_5)=(0,1)$, which is excluded by
Lemma \ref{lem:twoline}. Hence $\tau(\A)\in\{106,107\}$ and
$d_3\in\{10,9\}$. Finally, \eqref{eq:nk} and \eqref{eq:tauT} give
$n_3=\tau(\A)-78-3n_4-6n_5$ and $n_2=312-3\tau(\A)+3n_4+8n_5$, and the list
consists of the solutions with $n_2,n_3\ge0$, $n_5\ge1$,
$(n_4,n_5)\ne(0,1)$ and $n_4+3n_5\le3\tau(\A)-312$.
\end{proof}

\begin{corollary}\label{cor:m5nu}
Let $\A,\A'$ be a counterexample as in \eqref{eq:C2} with $d=13$ and $m=5$,
and let $d_3\in\{9,10\}$ be as in Proposition \ref{P:m5}. Then
$\nu(\A)=d_3-7$ and $\nu(\A')=d_3-8$, in particular $\nu(\A)=\nu(\A')+1$.
More precisely, the following hold.
\begin{enumerate}[label=\textnormal{(\roman*)}]
\item If $d_3=9$, that is $\tau(\A)=107$, then $\nu(\A)=2$, $r'=6$ and
$\nu(\A')=1$, so $\A'$ is nearly free with exponents $(6,7)$.
\item If $d_3=10$, that is $\tau(\A)=106$, then $\nu(\A)=3$, $\nu(\A')=2$, and
one of the following holds:
\begin{enumerate}[label=\textnormal{(\alph*)}]
\item $r'=6$ and $\A'$ is a minimal POG arrangement with exponents $(6,7,8)$;
\item $r'=7$ and $\A'$ is a maximal Tjurina arrangement of type $(13,7)$, with
exponents $(7,7,7,7)$.
\end{enumerate}
\end{enumerate}
\end{corollary}

\begin{proof}
By Proposition \ref{P:m5} we have $\tau(\A')=\tau(\A)=116-d_3$. Since
$r=5<11/2$, the first case of \eqref{eq:nu} gives
$\nu(\A)=\tau(13,5)_{\max}-\tau(\A)=109-\tau(\A)=d_3-7$, in accordance with
$\nu(\A)=d_3-d_2+1$. Since $r'\ge r+1=6>11/2$, the second case of
\eqref{eq:nu} gives $\nu(\A')=\ceil*{3\cdot12^2/4}-\tau(\A')=108-\tau(\A)=d_3-8$.

Next, $\tau(13,6)_{\max}=108$, $\tau(13,7)'_{\max}=109-\binom32=106$ and
$\tau(13,8)'_{\max}=112-\binom52=102$. As the du Plessis--Wall bound is strictly
decreasing in $r$, see Section \ref{sec:prelim}, \eqref{eq:dpwp} gives
$\tau(\A')\le106$ if $r'\ge7$ and $\tau(\A')\le102$ if $r'\ge8$. Hence $r'=6$
if $\tau(\A)=107$, and $r'\in\{6,7\}$ if $\tau(\A)=106$.

(i) Here $\nu(\A')=1$, so $\A'$ is nearly free \cite{Dfree,DSnf}, with
exponents $(r',d-r')=(6,7)$.

(ii) If $r'=6$, then $\tau(\A')=106=\tau(13,6)_{\max}-2$, and $\A'$ is a
minimal POG arrangement with exponents $(r',d-r',d-r'+1)=(6,7,8)$ by
\cite[Theorem 1.5]{DimcaSticlaruMPOG}. If $r'=7$, then $r'\ge d/2$ and
$\tau(\A')=106=\tau(13,7)'_{\max}$, so equality holds in \eqref{eq:dpwp} and
$\A'$ is a maximal Tjurina arrangement of type $(13,7)$ \cite{maxTjurina}. Its
module of Jacobian relations is then minimally generated by $2r'-d+3=4$
relations of degree $r'=7$, that is, $\A'$ has exponents $(7,7,7,7)$.
\end{proof}

We now show that seven of the fifteen numerical types of
Proposition \ref{P:m5} can be excluded by an elementary counting argument.

\begin{proposition}\label{P:m5bezout}
Let $\A,\A'$ be a counterexample as in \eqref{eq:C2} with $d=13$ and $m=5$.
Then every line of $\A$ passes through a point of multiplicity at least four,
and $(n_2,n_3,n_4,n_5)$ is none of the seven types
\[
(2,20,1,1),\ (5,19,1,1),\ (5,17,2,1),\ (10,16,0,2),\
(7,17,0,2),\ (10,14,1,2),\ (15,11,0,3).
\]
\end{proposition}

\begin{proof}
Let $L\in\A$. For $d=13$ and $m=5$, the second identity in \eqref{eq:delete},
that is, B\'ezout's theorem applied to $L$ and $\A\setminus\{L\}$, reads
\begin{equation}\label{eq:bez5}
n_2^L+2n_3^L+3n_4^L+4n_5^L=12 .
\end{equation}
If $L$ contains no point of multiplicity at least four, then
$n_2^L+2n_3^L=12$, while $n_2^L+n_3^L=r_L\le6$ by Proposition \ref{P:m5}.
Hence $n_3^L=6$ and $n_2^L=0$, that is, all singular points of $\A$ on $L$ are
triple points, which is impossible by Lemma \ref{lem:ziegler}. This proves the
first claim.

Consequently, the lines through the points of multiplicity four or five cover
$\A$. As there are four lines through a point of multiplicity four and five
lines through a point of multiplicity five, this requires $4n_4+5n_5\ge13$,
which excludes $(2,20,1,1)$, $(5,19,1,1)$, $(10,16,0,2)$ and $(7,17,0,2)$.

For the type $(5,17,2,1)$ we have $4n_4+5n_5=13$, so the pencils of lines
through the point $q$ of multiplicity five and the points $p_1,p_2$ of
multiplicity four are pairwise disjoint and form a partition of $\A$. Two lines
of the same pencil meet only at its centre, so every triple point lies on
exactly one line of each pencil. In particular, a triple point is the
intersection of a line through $p_1$ with a line through $p_2$, and distinct
triple points give distinct such pairs. Hence $n_3\le16<17$.

For the type $(10,14,1,2)$, let $p$ be the point of multiplicity four and
$q_1,q_2$ the points of multiplicity five. If $p,q_1,q_2$ lie on a line of
$\A$, the three pencils through them consist of $4+5+5-2=12<13$ lines. Otherwise
they consist of $14-j$ lines, where $j$ is the number of the lines $pq_1$,
$pq_2$, $q_1q_2$ belonging to $\A$. Since the pencils cover $\A$, we get
$j=1$. If this line is $q_1q_2$, then no line through $p$ passes through $q_1$
or $q_2$, so a line $L\ni p$ satisfies $n_2^L+2n_3^L=9$ by \eqref{eq:bez5}
and $n_2^L+n_3^L\le5$ by Proposition \ref{P:m5}, whence $n_3^L=4$. The four
lines through $p$ meet only at $p$, so they contain $16$ distinct triple
points, whereas $n_3=14$. Otherwise the line in $\A$ is $pq_i$ for some $i$,
say $pq_2$, and no line through $q_1$ passes through $p$ or $q_2$. By
\eqref{eq:bez5}, the $s$-th line through $q_1$ then contains $4-k_s$ triple
points and $2k_s$ double points for some integer $k_s\ge0$. As these five
lines meet only at $q_1$, we get $\sum_s(4-k_s)\le n_3=14$ and
$\sum_s2k_s\le n_2=10$, that is $\sum_sk_s\ge6$ and $\sum_sk_s\le5$, a
contradiction.

For the type $(15,11,0,3)$, \eqref{eq:sumr} gives
$\sum_{L\in\A}r_L=390-3\cdot107+3\cdot3=78$, so $r_L=6$ for every line $L$ by
Proposition \ref{P:m5}. Let $q_1,q_2,q_3$ be the points of multiplicity five.
If they lie on a line $\ell\in\A$, then \eqref{eq:bez5} gives
$n_2^\ell=n_3^\ell=0$, so $r_\ell=3$, a contradiction. Otherwise every line of
$\A$ contains at most two of the $q_i$, and if $j$ of the three lines $q_iq_k$
belong to $\A$, then the three pencils consist of $15-j$ lines. Since they
cover the thirteen lines of $\A$, we get $j=2$, say $q_1q_2,q_1q_3\in\A$ and
$q_2q_3\notin\A$. By \eqref{eq:bez5} and $r_L=6$, the lines $q_1q_2$ and
$q_1q_3$ contain no triple point, while every other line of $\A$ passes
through exactly one $q_i$ and contains exactly three triple points. The three
lines through a triple point pass through pairwise distinct points $q_i$, and
none of them is $q_1q_2$ or $q_1q_3$, so exactly one of them is one of the
three lines through $q_1$ other than $q_1q_2,q_1q_3$. Hence $n_3=3\cdot3=9$,
contradicting $n_3=11$.
\end{proof}

Proposition \ref{P:m5bezout} leaves eight of the fifteen numerical types of
Proposition \ref{P:m5}, namely $(8,16,2,1)$ and $(11,13,3,1)$ if
$\tau(\A)=106$, and $(8,14,3,1)$, $(11,11,4,1)$, $(14,8,5,1)$, $(17,5,6,1)$,
$(13,11,2,2)$ and $(16,8,3,2)$ if $\tau(\A)=107$. To settle them one would
have to extend the classification of Section \ref{sec:104} to configurations
containing blocks of size five, using the refined two-line condition of
Lemma \ref{lem:twoline} -- we leave this open.

\section{Two realizable lattices with a line through seven or more singular points}
\label{sec:extremal}

Arrangements of thirteen lines with $m=4$ and $\tau\in\{106,107\}$ having a
line through at least seven singular points are not needed for the proof, but
some of them are realizable, and we describe two of them. They show that the
bound $r_L\le6$ of Proposition \ref{P:thirteen} is not a consequence of the
numerical data $(d,m,\tau)$ alone: it uses the existence of the second member
$\A'$ of the pair, through Lemma \ref{lem:del}. In both examples the lines
through seven or more singular points have exactly $r_L=114-\tau$, the value
found in Remark \ref{rk:nolight}. By Proposition \ref{P:thirteen} neither
lattice can be the lattice of $\A$. Independently, their minimal degrees of
Jacobian relations are $7$ and $6$, see Subsection \ref{sub:verify}. We denote
the two arrangements by $\mathcal G_{106}$ and $\mathcal G_{107}$.

\subsection{On the verification}
\label{sub:verify}
The lattices and Tjurina numbers below were computed in
the relevant quadratic field $K$, represented as $\Q[u]/(u^2-c_1u-c_0)$ with
rational coefficients. By Lemma \ref{lem:galois}(i) the resulting value of $r$
is the one computed over $\C$. The Jacobian linear algebra was carried out modulo
primes $p$ in which the minimal polynomial of $u$ splits. Reduction modulo $p$
can only lower the rank of the coefficient matrix, so a vanishing $\AR(f)_r$
modulo $p$ forces $\AR(f)_r=0$ in characteristic zero, which is the direction
we need. The reverse inequality comes from \eqref{eq:dpwp} with no computation
at all:
\[
\tau=106>102=\tau(13,8)_{\max}'\ \Longrightarrow\ r\le7,
\qquad
\tau=107>106=\tau(13,7)_{\max}'\ \Longrightarrow\ r\le6 .
\]
Hence the values of $r$ stated below are proved.

\subsection{The arrangement \texorpdfstring{$\mathcal G_{106}$}{G106}}
\label{sec:106}
Let $u^2-u-1=0$, so that the field of definition is $\Q(\sqrt5)$. A dual
realization of $\mathcal G_{106}$, in which the point $p_i=(a_i:b_i:c_i)$ corresponds to the line
$a_ix+b_iy+c_iz=0$, is
\[
\begin{alignedat}{3}
p_0&=(1:0:0), &\quad p_1&=(0:1:0), &\quad p_2&=(0:0:1),\\
p_3&=(1:1:1), & p_4&=(0:1:u), & p_5&=(-1:-1:-u),\\
p_6&=(-u-1:-1:-u), & p_7&=(0:-u:-1), & p_8&=(1:0:1),\\
p_9&=(1:1:0), & p_{10}&=(-1:u-1:0), & p_{11}&=(-u-1:-1:-u-1),\\
p_{12}&=(-1:0:u-1). & & & &
\end{alignedat}
\]
Its seven quadruple points are
\[
\{0,4,5,6\},\ \{1,2,4,7\},\ \{1,3,8,11\},\ \{2,3,5,9\},\ \{3,6,7,10\},\
\{0,2,8,12\},\ \{0,1,9,10\},
\]
and its seven triple points are
\[
\{2,6,11\},\ \{3,4,12\},\ \{4,8,10\},\ \{5,7,8\},\ \{5,10,11\},\ \{6,8,9\},\
\{7,9,12\},
\]
so that $(n_2,n_3,n_4)=(15,7,7)$ and $\tau=106$. One has $\AR(f)_r=0$ for
$r\le6$, whence $r(\mathcal G_{106})=7$ by Subsection \ref{sub:verify}, and
$\mathcal G_{106}$ is maximal Tjurina of type $(13,7)$ with exponents
$(7,7,7,7)$, in accordance with $\dim\AR(f)_7=4$. Indeed,
$\tau(13,7)_{\max}'=144-35-3=106$, so this configuration lies exactly on the
maximal Tjurina boundary. The lines dual to $p_{11}$ and $p_{12}$ contain eight
singular points each, and every other line contains at most six.

The realization space of this matroid can be computed directly with the
method of Section \ref{sec:104}: fixing a projective frame and propagating the
incidences leaves a single parameter subject to a quadratic equation of
discriminant $5$, confirming the field of definition. The two roots give
Galois-conjugate realizations, which by Lemma \ref{lem:galois} have isomorphic
lattices and the same $r$, so this pair produces no jumping phenomenon.

\subsection{The arrangement \texorpdfstring{$\mathcal G_{107}$}{G107}}
\label{sec:107a}
Let $u^2+u+1=0$, a primitive cube root of unity, so that the field of
definition is $\Q(\sqrt{-3})$. A dual realization of $\mathcal G_{107}$ is
\[
\begin{alignedat}{3}
p_0&=(1:0:0), &\quad p_1&=(0:1:0), &\quad p_2&=(1:u:0),\\
p_3&=(0:0:1), & p_4&=(1:0:1), & p_5&=(0:1:1),\\
p_6&=(1:1:1), & p_7&=(1:u:u), & p_8&=(0:u:-1),\\
p_9&=(-1:-u:-1), & p_{10}&=(1:-1:0), & p_{11}&=(-1:0:u),\\
p_{12}&=(-1:1:u). & & & &
\end{alignedat}
\]
Its seven quadruple points are
\[
\{0,1,2,10\},\ \{0,3,4,11\},\ \{0,5,6,7\},\ \{1,3,5,8\},\ \{1,4,6,9\},\
\{2,3,7,9\},\ \{2,4,8,12\},
\]
and its eight triple points are
\[
\{1,11,12\},\ \{2,5,11\},\ \{3,10,12\},\ \{4,5,10\},\ \{5,9,12\},\
\{6,8,11\},\ \{7,8,10\},\ \{9,10,11\},
\]
so that $(n_2,n_3,n_4)=(12,8,7)$ and $\tau=107$. One has $\AR(f)_r=0$ for
$r\le5$, whence $r(\mathcal G_{107})=6$, and then \eqref{eq:nu} gives
$\nu(\mathcal G_{107})=108-107=1$, so $\mathcal G_{107}$ is nearly free with
exponents $(6,7)$, in accordance with $\dim\AR(f)_6=1$. The lines dual to
$p_6$, $p_7$ and $p_{12}$ contain seven singular points each, and every other
line contains at most six. After normalization of
a frame the realization space consists of two points, complex conjugate to
each other, so the two realizations have the same $r$ by
Lemma \ref{lem:galois}, and no deformation can produce a jump either.

\section{Degrees fourteen, fifteen and sixteen}
\label{sec:1516}

By Proposition \ref{P:rm} and Theorem \ref{corB} the next degrees to examine
are $14$, $15$ and $16$. The conditions $m\le r<(d-2)/2$, $d\ge2m+3$ and
$r\ge2d/m-2$ of Sections \ref{sec:prelim} and \ref{sec:first}, together with
$d\le3m+1$ when $r=m$ (Proposition \ref{P:rm}) and the exclusion of $r=m=4$,
leave the following triples $(d,r,m)$: for $d=14$ we have $r=5$ and $m\in\{4,5\}$,
for $d=15$ we have $r=5$ and $m=5$, or $r=6$ and $m\in\{4,5,6\}$. For $d=16$, we have $r=5$ and $m=5$, or $r=6$ and $m\in\{4,5,6\}$. On the side of $\A'$ we have
$r'\ge r+1$, so that
\begin{equation}\label{eq:tauup}
\tau(\A)=\tau(\A')\le\tau(d,r+1)'_{\max}
\end{equation}
by the monotonicity recalled in Section \ref{sec:prelim}, where
$\tau(d,k)'_{\max}$ stands for $\tau(d,k)_{\max}$ when $2k<d$. The equality forces
$r'=r+1$, and then $\A'$ is free if $2(r+1)<d$. For details regarding
computations, we refer again to \cite{Repo}.

\subsection{Deletion in degrees fourteen to sixteen}
\label{sub:del1416}
We first apply Lemma \ref{lem:del}, starting from the results in degree $13$.
For $d\in\{14,15,16\}$ and $m\le6$, formula \eqref{eq:sumr} reads
\begin{equation}\label{eq:sumr1416}
\sum_{L\in\A}r_L=5\binom d2-3\tau(\A)+n_4+3n_5+6n_6 ,
\end{equation}
and \eqref{eq:nk}, \eqref{eq:tauT} give, when $m\le5$,
\begin{equation}\label{eq:n2gen}
n_2=\binom d2-3\Bigl(\tau(\A)-\binom d2\Bigr)+3n_4+8n_5 ,
\qquad
n_3=\tau(\A)-\binom d2-3n_4-6n_5 .
\end{equation}

\begin{proposition}\label{P:fourteen}
Let $\A,\A'$ be a counterexample as in \eqref{eq:C2} with $d=14$. Then $r=5$,
$m\in\{4,5\}$, $\tau(\A)\in\{125,126\}$, every line of $\A$ contains at most
six singular points of $\A$, and
\[
n_4+3n_5\ \le\ 3\tau(\A)-371 .
\]
More precisely, $(n_2,\dots,n_m)$ is one of the following ten types:
\begin{enumerate}[label=\textnormal{(\alph*)}]
\item $m=4$: $(1,22,4)$ with $\tau(\A)=125$, and $(1,20,5)$, $(4,17,6)$,
$(7,14,7)$ with $\tau(\A)=126$;
\item $m=5$: $(0,25,1,1)$ with $\tau(\A)=125$, and $(0,23,2,1)$, $(3,20,3,1)$,
$(6,17,4,1)$, $(2,23,0,2)$, $(5,20,1,2)$ with $\tau(\A)=126$.
\end{enumerate}
If $m=4$, every line satisfies $n_3^L+2n_4^L\in\{7,8\}$. If $m=5$, then $\A$ is
POG with exponents $(5,9,d_3)$, where $d_3=137-\tau(\A)\in\{11,12\}$.
\end{proposition}

\begin{proof}
By Theorem \ref{thmA} we have $4\le m\le r\le5$, and $m=4$ forces
$r=\floor{11/2}=5$ and hence $r=5$ and $m\in\{4,5\}$. A counterexample of degree
$13$ has minimal degree of a Jacobian relation equal to $5$, by
Proposition \ref{P:thirteen}, so case (ii) of Lemma \ref{lem:del} cannot occur,
and $r_L\leq 6$ for every line. By \eqref{eq:sumr1416},
$455-3\tau(\A)+n_4+3n_5\le84$, which is the stated inequality. By
Lemma \ref{lem:twoline}, $n_4\ge3$ if $m=4$ and $n_4+3n_5\ge4$ if $m=5$. In both
cases $3\tau(\A)\ge374$, that is $\tau(\A)\ge125$. On the other hand,
\eqref{eq:tauup} gives $\tau(\A)\le\tau(14,6)_{\max}=127$, with equality only
if $\A'$ is free. Then $\A$ would be free by Terao's Conjecture in degree $14$
\cite{BK}, which is impossible, and hence $\tau(\A)\in\{125,126\}$.

By \eqref{eq:n2gen}, $n_2=364-3\tau(\A)+3n_4+8n_5$ and
$n_3=\tau(\A)-91-3n_4-6n_5$, and the list consists of the solutions with
$n_2,n_3\ge0$ and $n_4+3n_5\le3\tau(\A)-371$, subject to the conditions of
Lemma \ref{lem:twoline}. For $m=4$, the relation
$n_2^L+2n_3^L+3n_4^L=13$ of \eqref{eq:delete} gives $r_L=13-D_L$ and
$n_2^L=13-2D_L+n_4^L$, where $D_L=n_3^L+2n_4^L$. Hence $r_L\le6$ gives
$D_L\ge7$, while $n_2^L\ge0$ and $n_4^L\le4$ give $D_L\le8$. The last
assertion follows from Lemma \ref{lem:rm} and \eqref{eq:d3}, the free case
being excluded as in the proof of Proposition \ref{P:basic}.
\end{proof}

Proposition \ref{P:fourteen} places degree $14$ in the same situation as the
case $d=13$, $m=4$ of Section \ref{sec:thirteen}, i.e., finitely many numerical
types, and a sharp bound on the number of singular points of every line. The
method of Section \ref{sec:104} applies without change to the four types with
$m=4$, and after the extension to blocks of size five mentioned in
Subsection \ref{sub:m5} also to the six types with $m=5$. However, we have not carried out this classification. Note that the ten types have $n_2\le7$, so that
almost all pairs of lines meet in points of multiplicity at least three.

\begin{proposition}\label{P:fifteen}
Let $\A,\A'$ be a counterexample as in \eqref{eq:C2} with $d=15$.
\begin{enumerate}[label=\textnormal{(\roman*)}]
\item If $r=5$, then $m=5$, $\tau(\A)=148$, $\A$ is POG with exponents
$(5,10,12)$, $\A'$ is free with exponents $(6,8)$, every line of $\A$ contains
exactly six singular points of $\A$, and
\[
(n_2,n_3,n_4,n_5)\in\{(2,19,6,1),\ (1,22,3,2),\ (0,25,0,3)\}.
\]
\item If $r=6$, then every line $L\in\A$ satisfies $r_L\le154-\tau(\A)$.
Moreover, if $r_L\ge8$, then $\A\setminus\{L\},\A'\setminus\{L'\}$ is a
counterexample as in Proposition \ref{P:fourteen}, and
$r_L\in\{153-\tau(\A),\,154-\tau(\A)\}$. In particular, if $\tau(\A)=147$, then
every line of $\A$ contains at most seven singular points of $\A$.
\end{enumerate}
\end{proposition}

\begin{proof}
(i) We have $4\le m\le r=5$, and $m=4$ would force $r=\floor{12/2}=6$ and hence
$m=5$. By Proposition \ref{P:fourteen}, a counterexample of degree $14$ has
$r=5$, so case (ii) of Lemma \ref{lem:del} cannot occur and $r_L\le6$ for every
line. Then \eqref{eq:sumr1416} gives $525-3\tau(\A)+n_4+3n_5\le90$, that is
$n_4+3n_5\le3\tau(\A)-435$, while $n_2\ge0$ and \eqref{eq:n2gen} give
$3n_4+8n_5\ge3\tau(\A)-420$. Hence
\[
3\tau(\A)-420\ \le\ 3n_4+8n_5\ \le\ 3(n_4+3n_5)\ \le\ 9\tau(\A)-1305 ,
\]
so $\tau(\A)\ge148$. By \eqref{eq:tauup}, $\tau(\A)\le\tau(15,6)_{\max}=148$,
with equality only if $r'=6$ and $\A'$ is free, with exponents $(6,8)$. Thus
$\tau(\A)=148$, and Lemma \ref{lem:rm} and \eqref{eq:d3} show that $\A$ is POG
with exponents $(5,10,12)$. The three types are the solutions of
$n_4+3n_5\le9$ and $3n_4+8n_5\ge24$ with $n_5\ge1$. All of them have
$n_4+3n_5=9$, so that $\sum_Lr_L=90$ by \eqref{eq:sumr1416}, and every line
contains exactly six singular points.

(ii) The first inequality is Lemma \ref{lem:deltau}, since
$\tau(14,6)_{\max}=127$. If $r_L\ge8=r+2$, then case (i) of
Lemma \ref{lem:del} fails, so $\B,\B'$ is a counterexample of degree $14$ and
$\tau(\B)\in\{125,126\}$ by Proposition \ref{P:fourteen}. By
\eqref{eq:deleterL}, $r_L=\tau(\B)+28-\tau(\A)$.
\end{proof}

In particular, the case $(d,r,\tau(\A))=(15,5,147)$ of Table \ref{tab:1516}
below does not occur.

\begin{proposition}\label{P:sixteen}
Let $\A,\A'$ be a counterexample as in \eqref{eq:C2} with $d=16$. Then $r=6$.
Moreover:
\begin{enumerate}[label=\textnormal{(\roman*)}]
\item if $m=4$, then every line of $\A$ contains at most seven singular points
of $\A$, and $164\le\tau(\A)\le169$;
\item if $m\in\{5,6\}$ and a line $L\in\A$ contains at least eight singular
points of $\A$, then $\A\setminus\{L\},\A'\setminus\{L'\}$ is a counterexample
as in Proposition \ref{P:fifteen}\textnormal{(i)}, so that $r_L=178-\tau(\A)$ and
$\A'\setminus\{L'\}$ is free. In particular, this cannot happen if Terao's
Conjecture holds in degree $15$. If it does and $m=5$, then
$165\le\tau(\A)\le169$.
\end{enumerate}
\end{proposition}

\begin{proof}
Suppose first that $r=5$, then $m=5$ since $m=4$ forces $r=\floor{13/2}=6$.
Every counterexample of degree $15$ has $r\ge5$, by Theorem \ref{thmA} and
since $m=4$ forces $r=6$ in degree $15$, hence case (ii) of Lemma \ref{lem:del}
cannot occur, and $r_L\le6$ for every line. By \eqref{eq:sumr1416} and
\eqref{eq:n2gen},
\[
3\tau(\A)-480\ \le\ 3n_4+8n_5\ \le\ 3(n_4+3n_5)\ \le\ 9\tau(\A)-1512 ,
\]
so that $\tau(\A)\ge172$, contradicting $\tau(\A)\le\tau(16,6)_{\max}=171$.
Hence $r=6$.

Let now $L\in\A$ with $r_L\ge8=r+2$. By Lemma \ref{lem:del}, $\B,\B'$ is a
counterexample of degree $15$ with $r(\B)=5$ and $m(\B)\le m$, so by
Proposition \ref{P:fifteen}(i) we have $m(\B)=5$, $\tau(\B)=148$ and $\B'$
free. Moreover, $r_L=\tau(\B)+30-\tau(\A)=178-\tau(\A)$ by \eqref{eq:deleterL}.
This proves (ii), and shows that no such line exists when $m=4$. In that
case \eqref{eq:sumr1416} gives $600-3\tau(\A)+n_4\le112$, while $n_2\ge0$ gives
$n_4\ge\tau(\A)-160$ by \eqref{eq:n2gen}. Hence
$\tau(\A)-160\le3\tau(\A)-488$, that is $\tau(\A)\ge164$. If $m=5$ and every line has $r_L\le7$, then
\eqref{eq:sumr1416} gives $n_4+3n_5\le3\tau(\A)-488$, while $n_2\ge0$ gives
$3n_4+8n_5\ge3\tau(\A)-480$ by \eqref{eq:n2gen}. As $n_5\ge1$, we get
$3\tau(\A)-480\le3(n_4+3n_5)-n_5\le9\tau(\A)-1465$, so $\tau(\A)\ge165$. The upper bound is
$\tau(\A)\le\tau(16,7)_{\max}=169$ by \eqref{eq:tauup}, which completes the proof.
\end{proof}

\subsection{The case \texorpdfstring{$r=m$}{r=m}}
Here Lemma \ref{lem:rm} applies: as $\A$ is not free, it is POG with exponents
$(r,d-r,d_3)$, where $d-r\le d_3\le d-2$, and \eqref{eq:d3} gives
\[
\tau(\A)=\tau(d,r)_{\max}-(d_3-d+r+1).
\]
The bound \eqref{eq:tauup} then bounds $d_3$ from below, and
Table \ref{tab:1516} records the result. In the last column,
$\A'$ is described by the value of $\tau(\A')$ relative to the du
Plessis--Wall bound for each admissible $r'$. Recall that a minimal POG arrangement with exponents $(r',d-r',d-r'+1)$ is characterized by
$\tau=\tau(d,r')_{\max}-2$, see \cite{DimcaSticlaruMPOG}.

\begingroup
\small
\begin{longtable}{ccccl}
\caption{The numerical types in degrees $14$, $15$ and $16$ with $r=m$
allowed by \eqref{eq:tauup}. Here $r=m$ and $d_3$ refer to $\A$, which is POG
with exponents $(r,d-r,d_3)$, and the last column describes $\A'$, which has
$r'>r$. The rows marked $\dagger$ do not occur: $(14,5,127)$ by
Proposition \ref{P:fourteen}, $(15,5,147)$ by Proposition \ref{P:fifteen} and
$(16,5,171)$ by Proposition \ref{P:sixteen}.}
\label{tab:1516}\\
\toprule
$d$ & $r$ & $m$ & $d_3(\A)$, $\tau(\A)$ & possible types of $\A'$\\
\midrule
\endfirsthead
\toprule
$d$ & $r$ & $m$ & $d_3(\A)$, $\tau(\A)$ & possible types of $\A'$\\
\midrule
\endhead
\bottomrule
\endlastfoot
14 & 5 & $5$ & $10$, $127^\dagger$ & free $(6,7)$\\
14 & 5 & $5$ & $11$, $126$ & nearly free $(6,8)$; nearly free $(7,7)$, i.e.\ maximal Tjurina $(14,7)$\\
14 & 5 & $5$ & $12$, $125$ & minimal POG $(6,8,9)$ or $(7,7,8)$\\
\midrule
15 & 5 & $5$ & $12$, $148$ & free $(6,8)$\\
15 & 5 & $5$ & $13$, $147^\dagger$ & nearly free $(6,9)$; free $(7,7)$\\
15 & 6 & $6$ & $9$, $147$ & free $(7,7)$\\
15 & 6 & $6$ & $10$, $146$ & nearly free $(7,8)$\\
15 & 6 & $6$ & $11$, $145$ & minimal POG $(7,8,9)$; maximal Tjurina $(15,8)$\\
15 & 6 & $6$ & $12$, $144$ & $r'=7$, $\tau=\tau(15,7)_{\max}-3$; $r'=8$, $\tau=\tau(15,8)'_{\max}-1$\\
15 & 6 & $6$ & $13$, $143$ & $r'=7$, $\tau=\tau(15,7)_{\max}-4$; $r'=8$, $\tau=\tau(15,8)'_{\max}-2$\\
\midrule
16 & 5 & $5$ & $14$, $171^\dagger$ & free $(6,9)$\\
16 & 6 & $6$ & $11$, $169$ & free $(7,8)$\\
16 & 6 & $6$ & $12$, $168$ & nearly free $(7,9)$; nearly free $(8,8)$, i.e.\ maximal Tjurina $(16,8)$\\
16 & 6 & $6$ & $13$, $167$ & minimal POG $(7,9,10)$ or $(8,8,9)$\\
16 & 6 & $6$ & $14$, $166$ & $r'=7$, $\tau=\tau(16,7)_{\max}-3$; $r'=8$, $\tau=\tau(16,8)'_{\max}-2$\\
\end{longtable}
\endgroup

The row $d=16$, $r=5$ lies exactly on the boundary $d=3m+1$ of
Proposition \ref{P:rm} -- it forces $m=5$, $d_3=14$, $\tau(\A)=171$ and $\A$
being POG with exponents $(5,11,14)$, and it is excluded by
Proposition \ref{P:sixteen}.

\subsection{The case \texorpdfstring{$r>m$}{r>m}}
Here $\A$ need not be POG, and the a priori lower bound is the du
Plessis--Wall inequality $\tau(\A)\ge\tau(d,r)_{\min}$ of \eqref{eq:dpw},
which gives $104$, $112$ and $135$ for $d=14,15,16$ respectively. In two of the four cases,
and in a third one conditionally on Terao's Conjecture in degree $15$, it is
superseded by the bounds of Section \ref{sub:del1416}. Table \ref{tab:1516b}
records the resulting intervals.

\begingroup
\small
\begin{longtable}{cccl}
\caption{The admissible values of $\tau(\A)$ in degrees $14$, $15$ and $16$
with $r>m$. The bound marked $\ast$ assumes Terao's Conjecture in degree $15$;
unconditionally one has $135\le\tau(\A)$ there.}
\label{tab:1516b}\\
\toprule
$d$ & $r$ & $m$ & $\tau(\A)$\\
\midrule
\endfirsthead
\toprule
$d$ & $r$ & $m$ & $\tau(\A)$\\
\midrule
\endhead
\bottomrule
\endlastfoot
14 & 5 & $4$ & $125\le\tau(\A)\le126$ (Proposition \ref{P:fourteen}); $\A'$ nearly free if $\tau(\A)=126$\\
15 & 6 & $4, \, 5$ & $112\le\tau(\A)\le147$; $\A'$ free $(7,7)$ if $\tau(\A)=147$\\
16 & 6 & $4$ & $164\le\tau(\A)\le169$ (Proposition \ref{P:sixteen}); $\A'$ free $(7,8)$ if $\tau(\A)=169$\\
16 & 6 & $5$ & $165^\ast\le\tau(\A)\le169$; $\A'$ free $(7,8)$ if $\tau(\A)=169$\\
\end{longtable}
\endgroup

\subsection{Cases depending on Terao's Conjecture}

\begin{proposition}\label{P:terao}
In each of the three cases where $(d,r,\tau(\A))$ is one of $(15,5,148)$,
$(15,6,147)$, $(16,6,169)$, the arrangement $\A'$ is forced to be free, and a
counterexample would contradict Terao's Conjecture in degree $d$. The two
further cases $(14,5,127)$ and $(16,5,171)$ in which $\A'$ would be forced to
be free do not occur, by \cite{BK} and Proposition \ref{P:sixteen}
respectively.
\end{proposition}

\begin{proof}
In those cases $\tau(\A')=\tau(d,r')_{\max}$, so $\A'$ is free with exponents
$(r',d-1-r')$. As $\LL(\A)\cong\LL(\A')$, Terao's Conjecture in degree $d$ would
make $\A$ free with the same exponents, whence $r(\A)=r'$ contradicting
$r(\A)<r'$. Terao's Conjecture is known to hold for $d\le14$ by \cite{BK}.
\end{proof}

Propositions \ref{P:fourteen}, \ref{P:fifteen}, \ref{P:sixteen} and
\ref{P:terao} prove Theorem \ref{thmC}. For $d=15$ and $d=16$, Terao's
Conjecture is open, so Proposition \ref{P:terao} unfortunately only locates the difficulty rather than removing it. The same is true of the cases in which $\A'$ is nearly free, namely $\tau(\A)=126$ for $d=14$, $\tau(\A)=146$ for $d=15$ and $\tau(\A)=168$ for $d=16$. In both tables they would be settled by the analogue of Terao's Conjecture for nearly free arrangements were it known in
those degrees, compare \cite{AIM,DIM2}. All remaining cases are independent of
both conjectures, since there $\nu(\A')\ge2$. In Table \ref{tab:1516} these
are $d=14$ with $\tau(\A)=125$, $d=15$ with $\tau(\A)\in\{143,144,145\}$ and
$d=16$ with $\tau(\A)\in\{166,167\}$. In Table \ref{tab:1516b} they are the
values $\tau(\A)\le\tau(d,r+1)_{\max}-2$, for which
$\nu(\A')=\ceil*{3(d-1)^2/4}-\tau(\A)\ge2$ because $r'\ge r+1\ge(d-2)/2$.

\subsection{A criterion of Abe and Dimca}
\label{sub:charpoly}

When Terao's Conjecture is not available one may instead invoke
\cite[Theorem 1.5]{AbeDimcaSplit}: if a line arrangement $\A$ has a line
containing at most four intersection points of $\A$, and that line carries no
point of multiplicity at least $|\A|/2$, then whether $\nu(\A)\le1$, that is
whether $\A$ is free or nearly free, is determined by the reduced
characteristic polynomial $\chi(\A;t)$ of $\A$ of degree two, and if
$\chi(\A;t)$ is not a perfect square, it determines which of the two holds.
Both hypotheses are conditions on $\LL(\A)$, so they hold for $\A$ if and only
if they hold for $\A'$. This yields the following result.

\begin{proposition}\label{P:charpoly}
Let $\A,\A'$ be a counterexample and suppose that some line of $\A$ contains at
most four intersection points of $\A$. Then $\A'$ is not nearly free, and $\A'$
is free only if
\[
d=2r+3,\qquad r'=r+1,\qquad \chi(\A;t)=(t-r')^2 .
\]
\end{proposition}

\begin{proof}
Since $d\ge2m+3$ by Proposition \ref{P:basic}, no point has multiplicity at
least $d/2$, so the second hypothesis of \cite[Theorem 1.5]{AbeDimcaSplit} is
automatic and the criterion applies to $\A$ and to $\A'$ alike. As
$\chi(\A;t)=\chi(\A';t)$, we conclude that $\nu(\A)\le1$ if and only if
$\nu(\A')\le1$.

On the other hand, by \eqref{eq:nu} we have
$\nu(\A)=\tau(d,r)_{\max}-\tau(\A)$, and also
$\nu(\A')\le\tau(d,r+1)_{\max}-\tau(\A)$: if $r'<(d-2)/2$ this follows from
$r+1\le r'\le(d-1)/2$ and the monotonicity of $k\mapsto\tau(d,k)_{\max}$, and
otherwise from $\tau(d,r+1)_{\max}\ge\ceil*{3(d-1)^2/4}$. Hence
\begin{equation}\label{eq:nudiff}
\nu(\A)-\nu(\A')\ \ge\ \tau(d,r)_{\max}-\tau(d,r+1)_{\max}=d-2r-2\ \ge\ 1 .
\end{equation}
If $\A'$ were nearly free, then $\nu(\A)\ge2$ by \eqref{eq:nudiff},
contradicting $\nu(\A)\le1$. If $\A'$ is free, then $1\ge\nu(\A)\ge d-2r-2\ge1$,
so $\A$ is nearly free and $d=2r+3$. Moreover, equality holds in
\eqref{eq:nudiff}, so $\tau(\A')=\tau(d,r+1)_{\max}$. As $\A'$ is free,
$\tau(\A')=\tau(d,r')_{\max}$ with $r'\le(d-1)/2$, and strict monotonicity
gives $r'=r+1$. Then Terao's Factorization Theorem \cite[p. 145]{OrlikTerao},
applied to $\A'$, gives
$\chi(\A;t)=\chi(\A';t)=(t-r')\bigl(t-(d-1-r')\bigr)=(t-r')^2$, as
$d-1-r'=r+1=r'$.
\end{proof}

The hypothesis of Proposition \ref{P:charpoly} is a serious restriction on the
lattice, and not one that the numerical type can supply.

\begin{proposition}\label{P:nofour}
For every line arrangement $\A$ of $d$ lines one has
\begin{equation}\label{eq:sumrL}
\sum_{L\in\A}r_L\ \ge\ 5\binom d2-3\tau(\A),
\end{equation}
with equality if and only if $m(\A)\le3$. Hence, if
$\tau(\A)<\tfrac56 d(d-3)$, the average of $r_L$ over the lines of $\A$ exceeds
$5$, so the existence of a line carrying at most four intersection points
cannot be deduced from the numerical type by counting incidences. Moreover, if
$m(\A)=4$, then every
line of $\A$ carries at least $\lceil(d-1)/3\rceil$ intersection points.
\end{proposition}

\begin{proof}
The inequality and the equality case follow from \eqref{eq:sumr}. The right
hand side of \eqref{eq:sumrL} exceeds $5d$ precisely when
$\tau(\A)<\tfrac53\bigl(\binom d2-d\bigr)=\tfrac56 d(d-3)$. The last claim
follows from $\sum_k(k-1)n_k^L=d-1$ and $k-1\le3$.
\end{proof}

For $d=14,15,16$ the bound $\tfrac56 d(d-3)$ equals ${\sim 128.3}$, $150$ and
${\sim 173.3}$ respectively, while $\tau(\A)\le126$, $148$ and $169$
respectively by Section \ref{sub:del1416} and Tables \ref{tab:1516} and
\ref{tab:1516b}. Hence the hypothesis of Proposition \ref{P:charpoly} is never
forced by this count in our range. It is moreover unavailable when $m=4$, since
$\lceil (d-1)/3\rceil=5$ for these three degrees. More generally, since every
line satisfies $r_L\ge\lceil (d-1)/(m-1)\rceil$, that hypothesis can hold only
if $m\ge(d+3)/4$, that is only if $m\ge5$ when $d\in\{14,15,16\}$. In fact it
never holds in degree $14$, nor for $(d,r)=(15,5)$. Indeed, for the three types
of Proposition \ref{P:fifteen}(i) every line contains exactly six singular
points. For the types of Proposition \ref{P:fourteen}(b), \eqref{eq:sumr1416}
gives $\sum_Lr_L\ge82$ while $r_L\le6$, so a line with at most four
intersection points can only occur for the type $(0,23,2,1)$, where
$\sum_Lr_L=82$, but on such a line $\sum_k(k-1)n_k^L\le4+3+3+2=12<13$. It can
therefore be tested on individual lattices only in the rows of
Table \ref{tab:1516} with $r=6$, and in the rows of Table \ref{tab:1516b} with
$m=5$.

Subject to that hypothesis, Proposition \ref{P:charpoly} \textbf{removes} every
case of Tables \ref{tab:1516} and \ref{tab:1516b} with $r=6$ in which $\A'$ is
free or nearly free, with a single exception: $d=15$, $r=6$, $\tau(\A)=147$, where
$\A$ is nearly free with exponents $(6,9)$, and indeed
$d=2r+3$, $r'=7=r+1$ and $\chi(\A;t)=(t-7)^2$ is a perfect square. Among
lattices having a line with at most four intersection points, that case is
thus the only one with $\A'$ free or nearly free which resists both
Proposition \ref{P:terao} and Proposition \ref{P:charpoly}. Without this
hypothesis, which by Proposition \ref{P:nofour} always fails when $m=4$, every
case with $\A'$ free and $d\in\{15,16\}$, and every case with $\A'$ nearly free,
is left open by both propositions.

\section{Ziegler-type constructions and realization spaces}
\label{sec:obstruction}

\subsection{Ziegler-type constructions}
The classical mechanism producing lattice-isomorphic arrangements with distinct
$r$ is the following (cf. \cite{Di}). Choose $k$ points of $\PP^2$ and a graph $G$ on them, and
let the arrangement consist of the lines spanned by the edges of $G$. A vertex
of degree $e$ becomes a point of multiplicity $e$, and for a general position
of the points all other intersections are nodes, so the lattice depends only on
$G$. Specializing the points preserves the lattice but may create an extra
Jacobian relation.

\begin{example}\label{ex:ziegler}
Let $G=K_{3,3}$ on six points, so $d=9$ and every vertex has degree three, with
$n_2=18$, $n_3=6$ and $\tau=42$ for every admissible position. Then
\[
r(\A)=6\ \text{ for general points},
\qquad
r(\A)=5\ \text{ when the six points lie on a smooth conic}.
\]
This is Ziegler's example \cite{Zie}, and it shows that $r$ is not a
lattice invariant. It does not contradict Conjecture \ref{C1}: here
$(d-2)/2=7/2$ and both values of $r$ exceed it, so the second branch of
\eqref{eq:nu} applies to both members and gives $\nu=48-42=6$ in either case.
\end{example}

\begin{proof}[Proof of Theorem \ref{thmD}]
Let $V$ be the set of points of multiplicity at least three, and for $v\in V$
let $e_v$ be its multiplicity, so that $3\le e_v\le m$. The lines of $\A$
through two points of $V$ are the edges of a simple graph $G$ on $V$. Let $F$
be their number and $H$ the number of lines through exactly one point of $V$,
so that $F+H\le d$ and $\sum_ve_v=2F+H$. Hence
\[
\sum_ve_v-F_{\max}(e)\ \le\ \sum_ve_v-F\ =\ F+H\ \le\ d,
\]
where $F_{\max}(e)$ is the largest number of edges of a simple graph on $V$ in
which every vertex $v$ has degree at most $e_v$. By \eqref{eq:tauT},
\[
\tau(\A)=\binom d2+\sum_v\binom{e_v-1}{2}.
\]
Maximizing over all sequences with $3\le e_v\le m$ and
$\sum_ve_v-F_{\max}(e)\le d$ gives $121$, $141$ and $160$ for $d=14,15,16$,
attained for instance at $(5^4,4^2)$, $(5^6)$ and $(6,5^5)$. It gives $112$,
attained at $(4^7)$, for $d=14$ and $m\le4$, and $156$, attained at $(5^6)$, for
$d=16$ and $m\le5$ -- see \cite{Repo}. The values of $\tau(\A)$ in
Proposition \ref{P:fourteen} and in Table \ref{tab:1516} are at least $125$,
$143$ and $166$ for $d=14,15,16$, so none of them is attained, and the same
holds for the values in Table \ref{tab:1516b} above the stated bounds.
\end{proof}

\begin{remark}
The bound $160$ is sharp: take six general points, the fifteen lines joining
them, and one further general line through one of the six points. If every
line passes through \emph{exactly} two points of $V$, then $(e_v)$ is
graphical with $\sum_ve_v=2d$, and the maxima become $158$ for $d=16$ and
$154$ for $d=16$, $m\le5$, the other values being unchanged.
\end{remark}

\begin{remark}
Theorem \ref{thmD} says that a counterexample in degrees $14$--$16$ with $r=m$,
or with $d=14$, or with $r>m$ and $\tau(\A)>141$, $156$ for $d=15,16$
respectively, must have a line passing through three or more points of
multiplicity at least three. In particular, it cannot arise from the
construction above.
\end{remark}

\subsection{Expected dimension of realization spaces}
\label{sub:dim}

Theorem \ref{thmD} concerns one specific construction. A cruder measure of
whether a lattice can move in a family is the expected dimension of its
realization space. The following lemma is well known to the experts, but we recall it for the completeness of our presentation.

\begin{lemma}\label{lem:edim}
Modulo $\mathrm{PGL}(3,\C)$, the realization space of $\LL(\A)$ in the
space of $d$-tuples of lines has dimension at least $e(\A)$ at each of its
points, where
\begin{equation}\label{eq:edim}
e(\A):=2d-8-\sum_{p\,:\,m_p\ge3}\bigl(m_p-2\bigr).
\end{equation}
\end{lemma}

\begin{proof}
Let $X\subset(\PP^{2\vee})^d$ be the set of $d$-tuples of lines realizing
$\LL(\A)$. The ambient space is smooth of dimension $2d$. Fix a point
$(\ell_1,\dots,\ell_d)$ of $X$. For each point $p$ of multiplicity $m_p\ge3$
choose two of the lines through $p$. In a neighborhood of the fixed point
these two lines remain distinct, their intersection point $p(\ell)$ depends
regularly on the lines, and the condition that each of the remaining $m_p-2$
lines through $p$ passes through $p(\ell)$ is a single equation. Since the
absence of further incidences and the distinctness of the lines are open
conditions, $X$ is defined near the fixed point by these
$N=\sum_p(m_p-2)$ equations, so every irreducible component of $X$ through
that point has dimension at least $2d-N$. The orbits of
$\mathrm{PGL}(3,\C)$ have dimension at most $8$, and this leads to \eqref{eq:edim}.
\end{proof}

The two sums in \eqref{eq:tauT} and \eqref{eq:edim} pull in opposite
directions: a point of multiplicity $m$ contributes $\binom{m-1}{2}$ to
$\tau(\A)$ but $m-2$ to the codimension, and the ratio
$\binom{m-1}{2}/(m-2)=(m-1)/2$ increases with $m$. Reaching the large values of
$\tau(\A)$ required in Section \ref{sec:1516} forces the multiplicity to be
concentrated, and concentrating it destroys the dimension. A genuine necessary
condition, valid for every arrangement, limits how far this concentration can
go.

\begin{proposition}\label{P:incid}
Let $k$ be the number of points of $\A$ of multiplicity at least three and put
$S=\sum_{m_p\ge3}m_p$. Then
\begin{equation}\label{eq:incid}
S\,(S-d)\ \le\ d\,k\,(k-1),
\end{equation}
and equality forces every line of $\A$ to carry exactly $S/d$ points of
multiplicity at least three.
\end{proposition}

\begin{proof}
For a line $L$ let $s_L$ be the number of points of multiplicity at least three
on $L$, so that $\sum_Ls_L=S$. Two such points span a unique line, whence
$\sum_L\binom{s_L}{2}\le\binom k2$. As $x\mapsto\binom x2$ is convex, Jensen's
inequality gives $\sum_L\binom{s_L}{2}\ge d\binom{S/d}{2}$, that is
$S(S-d)/(2d)\le k(k-1)/2$. Equality in Jensen's inequality forces all $s_L$ to
coincide.
\end{proof}

We apply these two conditions to the numerical types left by
Section \ref{sec:1516}. Here a \emph{numerical type} is a vector
$(n_2,\dots,n_m)$ with the prescribed $d$ and $\tau(\A)$, whose maximal
multiplicity $m$ is one of the values allowed in the corresponding row, and
which satisfies \eqref{eq:hirz}, the conditions of Lemma \ref{lem:twoline}, and
the bound $\sum_Lr_L\le d\rho$ obtained from \eqref{eq:sumr1416} and the bound
$r_L\le\rho$ of Section \ref{sub:del1416}: $\rho=6$ for $d=14$ and for
$(d,r)=(15,5)$, $\rho=154-\tau(\A)$ for $(d,r)=(15,6)$, $\rho=7$ for
$(d,r,m)=(16,6,4)$, and no bound for $(d,r)=(16,6)$ with $m\in\{5,6\}$. The
column ``types'' of Table \ref{tab:dim} counts these -- the rows of the excluded
cases $(14,5,127)$, $(15,5,147)$ and $(16,5,171)$ are omitted.

\begingroup
\small
\begin{longtable}{cccccc}
\caption{Numerical types compatible with a positive expected dimension.}
\label{tab:dim}\\
\toprule
$d$ & case & $\tau(\A)$ & types & with $e(\A)\ge1$ & and satisfying \eqref{eq:incid}\\
\midrule
\endfirsthead
\toprule
$d$ & case & $\tau(\A)$ & types & with $e(\A)\ge1$ & and satisfying \eqref{eq:incid}\\
\midrule
\endhead
\bottomrule
\endlastfoot
14 & $r=m=5$ & $125$--$126$ & 6 & 0 & 0\\
14 & $r=5$, $m=4$ & $125$--$126$ & 4 & 0 & 0\\
\midrule
14 & total & & 10 & 0 & 0\\
\midrule
15 & $r=m=5$ & $148$ & 3 & 0 & 0\\
15 & $r=m=6$ & $143$--$147$ & 281 & 65 & 56\\
15 & $r=6$, $m\in\{4,5\}$ & $112$--$147$ & 808 & 327 & 325\\
\midrule
15 & total & & 1092 & 392 & 381\\
\midrule
16 & $r=m=6$ & $166$--$169$ & 379 & 38 & 22\\
16 & $r=6$, $m=4$ & $164$--$169$ & 33 & 0 & 0\\
16 & $r=6$, $m=5$ & $135$--$169$ & 964 & 325 & 324\\
\midrule
16 & total & & 1376 & 363 & 346\\
\end{longtable}
\endgroup

All ten types in degree $14$, the three types with $(d,r)=(15,5)$ and the
$33$ types with $(d,r,m)=(16,6,4)$ have $e(\A)\le0$, so the dimension count
gives no information about them. The surviving types with $r=m$ have
$6\le k\le12$, $29\le S\le45$ and $1\le e(\A)\le5$. Those with $r>m$ have
$2\le k\le20$, $9\le S\le63$ and $1\le e(\A)\le17$. The inequality
\eqref{eq:incid} is an honest obstruction, valid for every arrangement, but on
its own it is weak: applied to all types of Table \ref{tab:dim}, it removes
$9$ types in the row $(15,r=m=6)$, $3$ in the row $(15,r=6)$, $18$ in the row
$(16,r=m=6)$, $4$ in the row $(16,r=6,m=5)$, and none in the other rows. Its
equality case, however, is rigid enough to analyse by hand.

\begin{example}\label{ex:nearmiss}
Among the surviving types, exactly five satisfy \eqref{eq:incid} with
equality, all in degree $15$ and all with $k=6$ and $S=30=2d$: the types
$(n_2,\dots,n_6)=(43,0,2,2,2)$, $(42,0,3,0,3)$, $(42,1,0,3,2)$ and
$(41,1,1,1,3)$, with $r=m=6$ and $\tau(\A)=143$, $144$, $144$, $145$, and the
type $(n_2,\dots,n_5)=(45,0,0,6)$ with $m=5$ and $\tau(\A)=141$. By
Proposition \ref{P:incid} every line then carries exactly two of the six
multiple points, so these points are the vertices of a graph with fifteen
edges on six vertices, necessarily $K_6$, and every multiple point has
multiplicity $5$. This excludes the first four types (which are also excluded
by Theorem \ref{thmD}, as $\tau(\A)>141$). The last type is realized precisely
by the fifteen lines joining six points no three of which are collinear, in
positions where no further concurrences occur. For six general points one
finds $\AR(f)_8=0$, so that $r=9$ by \eqref{eq:dpwp}, since
$\tau(15,9)'_{\max}=141$: the general member is maximal Tjurina of type
$(15,9)$ and balanced. The same value $r=9$ is obtained when the six points
lie on a smooth conic, so this natural specialization produces no jump. A
counterexample of this type would require a special position of the six points
for which $r$ drops to $6$.
\end{example}

\begin{remark}\label{rk:edim}
The condition $e(\A)\ge1$ is sufficient for the realization space to be
positive-dimensional, provided it is non-empty, but not necessary: the
concurrency conditions may be dependent, as happens for symmetric
configurations, and then the actual dimension exceeds $e(\A)$.
Table \ref{tab:dim} therefore does not exclude a counterexample -- it only
identifies the $381$ and $346$ numerical types in degrees $15$ and $16$ whose
non-empty realization spaces are automatically positive-dimensional. It should
also be said that a jump of $r$ does not require a positive-dimensional family
at all: what is needed is a lattice whose realization space has two components
carrying different values of $r$, and Section \ref{sec:extremal} exhibits
lattices in degree $13$ whose realizations are isolated and defined over
quadratic fields. A rigid lattice with two such components would escape this
subsection entirely.
\end{remark}

\begin{remark}\label{rk:chmn}
Conjecture \ref{C1} is equivalent to a positive answer to
\cite[Question 7.12]{Cook+}, which asks whether the splitting type $(a,b)$ of
the derivation bundle is a combinatorial invariant. Indeed, by
\cite[Corollary 3.3]{AbeDimcaSplit} one has $\nu(\A)=(d-1)^2-\tau(\A)-ab$, which
is \eqref{eq:nu} rewritten, and $(d-1)^2-\tau(\A)$ is the second Chern class of
the derivation bundle, see also \cite[(7.1)]{Cook+}, hence a lattice
invariant. Since $a+b=d-1$, knowing $ab$ is the same as knowing $(a,b)$. This
is implicit in \cite[\S1]{AbeDimcaSplit}, where a positive answer to
\cite[Question 7.12]{Cook+} is described as a stronger form of Terao's
Conjecture.

The splitting type also governs unexpected curves. Let $Z$ be the set of $d$
points dual to $\A$. By \cite[Lemma 3.5(a) and Corollary 5.7]{Cook+}, $Z$
admits an unexpected curve if and only if $2a+2<d$, that is $b-a\ge2$, and no
$a+2$ points of $Z$ are collinear, that is $m(\A)\le a+1$. In this case $Z$ has
an unexpected curve of degree $j$ exactly for $a<j\le b-1$. For a
counterexample $\A,\A'$ one has $a=r\ge m$, so the dual configuration $Z$ of
$\A$ admits unexpected curves of degrees $r+1,\dots,d-r-2$, while the dual
configuration $Z'$ of $\A'$ has none of degree $r+1$, since $a'\ge r+1$. If
moreover $\A'$ is balanced, which is automatic when $d=2r+3$, for instance in
the case $d=13$, $r=5$ of Section \ref{sec:thirteen}, then $Z'$ admits no
unexpected curve at all. Thus a counterexample amounts to a pair of
lattice-isomorphic point configurations of which only one admits an unexpected
curve of degree $r+1$. Conversely, Conjecture \ref{C1} would imply that the
existence and the degrees of unexpected curves of $Z$ depend only on
$\LL(\A)$.
\end{remark}

\section{An extremal type in degree fifteen}
\label{sec:2114}

We illustrate how the types of Table \ref{tab:1516b} can be attacked directly.
Take $d=15$, $m=4$, $r=6$ and $\tau(\A)=147$. By \eqref{eq:tauT} one has
$n_3+3n_4=42$, and the extremal solution without triple points is
\[
(n_2,n_3,n_4)=(21,0,14).
\]
Every line then satisfies $n_2^L+3n_4^L=14$, and summing gives
$\sum_Ln_4^L=4n_4=56$, so that $\sum_L(4-n_4^L)=4$. Note that
Lemma \ref{lem:twoline} is vacuous here, since a line carries at most four of
the fourteen points of multiplicity four.

\begin{proposition}\label{P:2114}
There is no counterexample with $d=15$, $m=4$ and
$(n_2,n_3,n_4)=(21,0,14)$.
\end{proposition}

\begin{proof}
By Proposition \ref{P:fifteen}(ii) with $\tau(\A)=147$, every line satisfies
$r_L\le7$. As there are no triple points, $r_L=n_2^L+n_4^L=14-2n_4^L$, so
$n_4^L\ge4$, and $3n_4^L\le14$ gives $n_4^L=4$ for every line. Summing over
the fifteen lines gives $4n_4=60$, contradicting $n_4=14$.
\end{proof}

\begin{remark}\label{rk:baer}
The method of Section \ref{sec:104} can also be applied to this type. In the
extremal case in which a single line $L_0$ absorbs the whole defect, that is
$n_4^{L_0}=0$ and $n_4^L=4$ for the other fourteen lines (a case excluded at
once by Proposition \ref{P:2114}, since $r_{L_0}=14$), it proceeds as follows.
Dually, $L_0$ becomes a point on no block, while the remaining
fourteen points carry fourteen blocks of size four, each point lying on
exactly four of them, so that the uncovered pairs form a perfect matching. The
search produces $1920$ labelled solutions forming a single isomorphism class,
namely
\[
\begin{aligned}
&\{0,2,4,6\},\ \{0,3,8,10\},\ \{0,5,9,12\},\ \{0,7,11,13\},\ \{1,2,9,11\},\\
&\{1,3,5,7\},\ \{1,4,8,13\},\ \{1,6,10,12\},\ \{2,5,10,13\},\ \{2,7,8,12\},\\
&\{3,4,11,12\},\ \{3,6,9,13\},\ \{4,7,9,10\},\ \{5,6,8,11\}.
\end{aligned}
\]
This configuration is the complement of a Baer subplane $\PP^2(\F_2)$ in
$\PP^2(\F_4)$: the fourteen lines of $\PP^2(\F_4)$ meeting the subplane in
exactly one point carry four points of the complement each, and the seven
lines of the subplane give the perfect matching of uncovered pairs. Its
automorphism group has order $336$, in accordance with
$1920=2^7\cdot7!/336$. In particular, it is realizable over $\F_4$, and our
realizability test in characteristic zero, which leaves one parameter $s$ and an
ideal containing $s^2-s+1$ and $s^2-s-1$, returns the unit ideal, while in
characteristic two it returns $(s^2+s+1)$. Non-realizability over $\C$ can
also be seen very directly: the restriction to the points
$\{0,2,4,7,8,10,13\}$ is the Fano plane, with lines
$\{0,2,4\}$, $\{0,8,10\}$, $\{0,7,13\}$, $\{4,8,13\}$, $\{2,10,13\}$,
$\{2,7,8\}$, $\{4,7,10\}$. This is a small reminder that a matroid search must always
be completed by a realizability test in characteristic zero, and that the test
must verify the recovered lattice.
\end{remark}

\section*{AI Declaration}
Symbolic computations performed in this paper were designed and optimized with
the use of \texttt{Claude} (Anthropic). All the mathematical results were obtained using human intelligence by both authors. Piotr
Pokora acknowledges \texttt{Anthropic} for providing access to the Claude Team
plan for scientists.

\section*{Funding}
Alexandru Dimca is partially supported by the project ``Singularities and Applications'' -- CF 132/31.07.2023, funded by the European Union -- NextGenerationEU -- through Romania's National Recovery and Resilience Plan.

Piotr Pokora is supported by the National Science Centre (Poland) Sonata Bis Grant  \textbf{2023/\allowbreak 50/\allowbreak E/\allowbreak ST1/\allowbreak 00025.} For the purpose of Open Access, the authors have applied a CC-BY public copyright license to any Author Accepted Manuscript (AAM) version arising from this submission. 

\section*{Data Availability Statement}
The scripts supporting the computations in this paper are available in the repository \cite{Repo}.

\section*{Conflict of Interest}
The authors declare that they have no conflicts of interest related to this manuscript.

\end{document}